\documentclass[11pt,a4paper,reqno]{amsart}

\usepackage{amsmath,amsfonts,amsthm,epsfig,graphicx}
\usepackage{graphicx,color}

\usepackage[numbers]{natbib}

 \allowdisplaybreaks
\numberwithin{equation}{section}

\def\H{\mathcal H}

\def\R{\mathbb R}
\def\N{\mathbb N}

\def\dist{\hbox{dist}}
\def\e{\varepsilon}
\def\s{\sigma}
\def\S{\Sigma}
\def\vphi{\varphi}
\def\Div{{\rm div}\,}
\def\om{\omega}
\def\l{\lambda}
\def\g{\gamma}
\def\k{\kappa}
\def\Om{\Omega}
\def\de{\delta}
\def\Id{{\rm Id}}

\def\spt{{\rm spt}}

\def\pa{\partial}

\def\F{\mathcal{F}}

\newcommand{\hd}{\mathrm{hd}}

\def\bd{{\rm bd}\,}

\def\X{\boldsymbol{\mathcal E}}

\def\cof{{\rm cof}\,}

\renewcommand{\a}{\alpha}

\newcommand{\D}{\Delta}
\renewcommand{\l}{\lambda}

\newcommand{\Lip}{{\rm Lip}}

\renewcommand{\t}{\tau}

\renewcommand{\o}{\omega}

\renewcommand{\Div}{{\rm div \,}}
\newcommand{\ov}{\overline}

\newcommand{\diam}{\mathrm{diam}}
\newcommand{\cc}{\subset\subset}

\def\bd{{\rm bd}\,}
\def\INT{{\rm int}\,}

\def\PHI{\mathbf{\Phi}}
\def\PSI{\mathbf{\Psi}}

\def\C{\mathbf{C}}
\def\D{\mathbf{D}}

\theoremstyle{plain}
\newtheorem{theorem}{Theorem}[section]
\newtheorem{lemma}[theorem]{Lemma}

\newtheorem{proposition}[theorem]{Proposition}
\newtheorem{notation}[theorem]{Notation}
\newtheorem{definition}[theorem]{Definition}
\newtheorem{remark}[theorem]{Remark}

\title{On the shape of capillarity droplets in a container}
\author{F. Maggi}
\author{C. Mihaila}

\begin{document}

\begin{abstract}
{\rm We provide a quantitative description of global minimizers of the Gauss free energy for a liquid droplet bounded in a container in the small volume regime.}
\end{abstract}

\maketitle

\section{Introduction} Our aim is to provide a quantitative description of capillarity droplets in a bounded container. We work in the classical setting of capillarity theory based on the minimization of Gauss free energy under a volume constraint. In this framework, denoting by $A$ and $E$ two open bounded sets with Lipschitz boundary in $\R^n$ with $E\subset A$ (so that $E$ is the region occupied by a liquid droplet inside a container $A$), one looks for volume-constrained global/local minimizers or stable/stationary points of the energy
\[
\F_{A,\s}(E)+\int_E\,g(x)\,dx\,,
\]
where $g:A\to\R$ is a bounded potential energy density and where $\F_{A,\s}$ is the surface tension energy of the droplet,
\begin{equation}
  \label{surface energy}
  \F_{A,\s}(E)=\H^{n-1}(A\cap\pa E)+\int_{\pa A\cap\pa E}\s\,d\H^{n-1}\,.
\end{equation}
Here $\H^{n-1}$ denotes the $(n-1)$-dimensional Hausdorff measure in $\R^n$, so that $\H^{n-1}(A\cap\pa E)$ accounts for the surface tension energy of liquid/air internal interface of the droplet, while $\s:\pa A\to(-1,1)$ is a given function, modeling the relative adhesion coefficient between the liquid droplet and the solid walls of the container. In this way, $\int_{\pa A\cap\pa E}\s\,d\H^{n-1}$ accounts for the total surface tension energy of the liquid/solid boundary interface. Typically, one considers liquid droplets under the action of gravity, a situation that corresponds to taking $n=3$ and $g(x)=c\,x_n$ for a positive constant $c$.

There is a rich variational theory concerning the functional $\F_{A,\s}$. A portion of the classical literature on the problem assumes the existence and the smoothness of minimizers, and starting from these assumptions moves to their qualitative description. An excellent overview on this family of results is provided by the book of Finn \cite{Finn}. Existence theories have to be formulated in the setting of Geometric Measure Theory, and a particularly suitable framework is that of sets of finite perimeter and functions of bounded variation, see, e.g. \cite[Chapter 19]{maggiBOOK}. Since the boundary $\pa E$ of a minimizer of finite perimeter $E$ will just be a countable union of compact subsets of $C^1$-hypersurfaces, then one faces the problem of showing additional regularity for minimizers, which is a crucial step for understanding the relation of the mathematical model with the physical world.

The regularity issue can be trivialized by exploiting symmetries. For example, in the sessile and pendant liquid droplet problems it can be proved that minimizers are rotationally symmetric, a property which easily implies smoothness; see for example \cite{GonzalezREGOLARITAGOCCIA}. In the case of a generic container one cannot exploit symmetry to simplify the regularity problem. This regularity problem has then drawn the attention of several authors \cite{taylor78,caffarellifriedman85,gruterjost,gruter,gruter2,gruter3,luckhaus,caffarellimellet,dephilippismaggiCRELLE,dephilippismaggiARMA}.

For example, if $A$ has boundary of class $C^{1,1}$, $g$ is bounded and $\s$ is a Lipschitz function, then for each volume-constrained local minimizer $E$ of $\F_{A,\a}$ there exists a closed subset $\S$ of $M=\ov{A\cap\pa E}$ such that $M\setminus\S$ is a $C^{1,\a}$-hypersurface with boundary (for every $\a\in(0,1)$). Moreover, the set of boundary points $\bd(M\setminus\S)$ of $M\setminus\S$ is contained in $\pa A$, and if $\nu_G$ denotes the outer unit normal to $G\subset \R^n$, then {\it Young's law} holds, i.e.
\[
\nu_A(x)\cdot\nu_E(x)=\s(x)\,,\qquad\forall x\in \bd(M\setminus\S)\,.
\]
Finally, $\S$ is empty if $n\le 3$, it is discrete if $n=4$ and it has Hausdorff dimension less than $n-4$ otherwise. This result was proved in the case $n=3$ by Taylor \cite{taylor78}, and in higher dimension by Luckhaus \cite{luckhaus}. Luckhaus' result was unawares rediscovered as \cite[Corollary 1.4]{dephilippismaggiARMA} in a study on Young's law for anisotropic surface energies. (We also notice that in the above papers the regularity result is obtained with $\a=1/2$, although one can improve this to every $\a<1$ by boundary elliptic regularity.)

Capillarity phenomena are characterized by the dominance of the surface tension energy on the bulk/potential energy term. This is typically the case when the volume parameter $m$ is suitably small with respect to the various data of the problem and then the surface energy is order $m^{(n-1)/n}>>m$, while the potential energy is of order $m$. In the case where there is no container, the problem of describing global minimizers in the small volume regime has been addressed in \cite{FigalliMaggiARMA}. There it is shown that if $g$ is a locally bounded potential energy density with the property that $g(x)\to\infty$ as $|x|\to\infty$ (in particular, the gravitational potential is not allowed in this simplified model),  then there exists $m_0=m_0(n,g)>0$ such that every minimizer $E_m$ in
\[
\inf\Big\{P(E)+\int_E\,g(x)\,dx:|E|=m\Big\}\,,
\]
with $m\le m_0$ is connected and satisfies, for some $y_m\in\R^n$,
\begin{equation}
  \label{figallimaggi1}
  \Big|\Big(\frac{E_m-y_m}{m^{1/n}}\Big)\Delta \frac{B}{\om_n^{1/n}}\Big|\le C(n,g)\,m^{1/2n}\,,
  \qquad
  \hd\Big(\frac{\pa E_m-y_m}{m^{1/n}},\frac{\pa B}{\om_n^{1/n}}\Big)\le C(n,g)\,m^{1/n^2}\,,
\end{equation}
where $B=\{x\in\R^n:|x|<1\}$, $\om_n=|B|$, and $\hd(X,Y)$ denotes the Hausdorff distance between $X,Y\subset\R^n$, see \eqref{hd def} below. Moreover, if $g\in C^{1,\a}_{{\rm loc}}(\R^n)$ for some $\a>0$, then $\pa E_m$ is a $C^2$-hypersurface whose second fundamental form $\nabla\nu_{E_m}(x)$ at $x\in T_x\pa E_m$, after a suitable rescaling, is uniformly close to a suitable multiple of the identity tensor on $T_x\pa E_m$, with the quantitative estimate
\begin{equation}
  \label{figallimaggi2}
  \max_{x\in\pa E_m}\|m^{1/n}\,\nabla\nu_{E_m}(x)-c(n)\,\Id_{T_x\pa E_m}\|\le C(n,g)\,m^{2/(n+2)}\,.
\end{equation}
In particular, $E_m$ is convex. What is crucial here is the quantitative aspect of \eqref{figallimaggi1} and \eqref{figallimaggi2}. We are not only asserting a proximity result, we are also quantifying the distance (measured in various ways) from being a ball. In passing, let us also mention that this kind of analysis has been recently extended to the case of local minimizers, and even of stationary points, in \cite{ciraolomaggi}.

Returning to the case where a container $A$ is present, our goal here is to quantitatively describe the shape of global minimizers $E_m$ of
\begin{equation}
  \label{variational problem m}
  \g(m)=\inf\Big\{\F_{A,\s}(E)+\int_E\,g(x)\,dx:E\subset A\,,|E|=m\Big\}\,,
\end{equation}
in the small volume regime. When $g$ and $\s$ vanish identically, then \eqref{variational problem limiting} reduces to the well-known relative isoperimetric problem in $A$, and global minimizers in \eqref{variational problem limiting} are called {\it isoperimetric regions in $A$}. In this case Fall \cite{fall} has shown that isoperimetric regions converge, as $m\to 0^+$, to boundary points of $A$ where the mean curvature of $\pa A$ achieves its maximum. When $A$ is a convex polytope (and again $g$ and $\s$ vanish identically) then an analogous result was obtained by Ritor\'e and Vernadakis, who showed that isoperimetric regions converge to vertices of $\pa A$ with smallest solid angle \cite{ritorevernadakis}.

We thus expect, also the general case where $\s$ and $g$ are non-trivial, that $E_m$ should have small diameter and that it should concentrate around a boundary point of $A$. Specifically, (global) energy minimization should favor those points in $\pa A$ where $\s$ achieves its minimum value
\[
\s_0=\min_{\pa A}\s\in(-1,1)\,,
\]
and correspondingly $E_m$ should be contained in a ball of radius $O(m^{1/n})$ and center $p_m\in\pa A$ such that $\s(p_m)\to\s_0$ as $m\to 0^+$. In particular, in the small volume regime and if $\s$ has a strict minimum point on $\pa A$, then neither the potential energy nor the mean curvature of $\pa A$ should play a role in determining the asymptotic position of global minimizers.

Next, assuming this concentration at the boundary to happen, we would expect the blow-ups
\[
\frac{E_m-p_m}{m^{1/n}}
\]
to converge (as $m\to 0$) to minimizers in the sessile droplet problem with no gravity and constant adhesion coefficient $\s_0$. To be more precise, let us fix a reference half-space $H$ to be
\[
H=\{x\in\R^n:x_n>0\}\,.
\]
and given $\tau\in(-1,1)$ let us consider the variational problem
\begin{equation}
  \label{variational problem limitingx}
  \psi(\tau)=\inf\Big\{\F_{H,\tau}(F):F\subset H\,,|F|=1\Big\}\,.
\end{equation}
If we set
\begin{equation}
  \label{def K}
S(\tau)=\{x\in B:x_n>-\tau\}\,,
\qquad
K(\tau)=\tau\,e_n+\frac{S(\tau)}{|S(\tau)|^{1/n}}\,,
\end{equation}
then $\{z+K(\tau):z\in\pa H\}$ is the family of all minimizers in \eqref{variational problem limitingx}, see e.g. \cite[Theorem 19.21]{maggiBOOK}. We thus expect $(E_m-p_m)/m^{1/n}$ to converge in some sense to $K(\s_0)$. Our main result, which is illustrated in \begin{figure}
  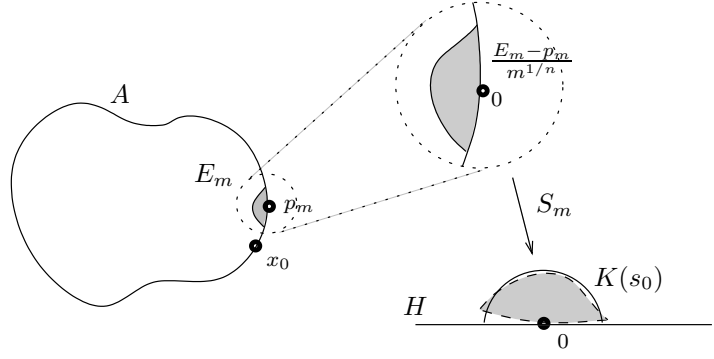\caption{{\small The situation in Theorem \ref{thm main}. A global minimizer $E_m$ is confined in a ball of radius $O(m^{1/n})$ centered at a point $p_m\in\pa A$ such that $\s(p_m)-\s_0=O(m^{1/n})$. In particular, if $\s_0=\s(x_0)$ for a unique $x_0\in\pa\Om$, then $p_m$ has to converge to $x_0$ with a velocity that depends on how fast $\s$ detaches from its minimum value at $x_0$. After a blow-up at $p_m$ to the scale $m^{1/n}$, and after a suitable rigid motion $S_m$, the global minimizer is Hausdorff close to the ideal droplet $K(\s_0)$, and the part of its boundary interior to $A$ is actually $C^{1,\a}$-diffeomorphic (for every $\a\in(0,1)$) to the spherical cap $H\cap\pa K(\s_0)$.}}
  \label{fig uno}
\end{figure}
 Figure \ref{fig uno}, proves that this converge happens in a global $C^{1,\a}$-sense, and it also quantifies the rate of convergence in Hausdorff distance.

\begin{theorem}\label{thm main}
  Let $n\ge 2$, $A$ be a bounded open connected set with boundary of class $C^{1,1}$, let $\s\in{\rm Lip}(\pa A)$ with $-1<\s(x)<1$ for every $x\in\pa A$, and let $g\in{\rm Lip}(A)$. Then there exist positive constants $C_0$ and $m_0$ depending on $A$, $\s$ and $g$ with the following property.

  If $E_m$ is a minimizer in \eqref{variational problem m} with $m\le m_0$, then there exists $p_m\in\pa A$ such that
  \begin{equation}
    \label{main theorem small diameter}
    E_m\subset B_{p_m,C_0\,m^{1/n}}\qquad 0\le \s(p_m)-\s_0\le C_0\,m^{1/n}\,,
  \end{equation}
  and, for a linear isometry $S_m$ of $\R^n$ and with $M_m=\ov{A\cap\pa E_m}$ and $M_0=\ov{H\cap\pa K(\s_0)}$,
  \begin{equation}
    \label{main theorem hd}
    \hd\Big(S_m\Big(\frac{M_m-p_m}{m^{1/n}}\Big),M_0\Big)\le C_0\,m^{1/2\,n^2}\,.
  \end{equation}
  Moreover, $M_m$ is a $C^{1,\a}$-hypersurface with boundary for every $\a\in(0,1)$, the set of boundary points $\bd(M_m)$ is contained in $\pa A$, and there exist a $C^{1,\a}$-diffeomorphism $f_m:M_0\to M_m$  and $y_m\in\pa A$ such that
  \begin{equation}
    \label{diffeo convergence}
    \inf\Big\{\|f_m-\big(y_m+m^{1/n}S\big)\|_{C^1(M_0)}:\mbox{$S:\R^n\to\R^n$ is an isometry}\Big\}=o(m^{1/n})\,,
  \end{equation}
  as $m\to 0^+$,  where this limit relation depends on $A$, $\s$ and $g$ only, but not on the family of minimizers $E_m$.
\end{theorem}

\begin{remark}\label{remark intro}
  {\rm In the course of the proof of Theorem \ref{thm main} we shall prove that
  \begin{equation}
    \label{gamma m espansione}  \g(m)=\psi(\s_0)\,m^{(n-1)/n}\,\Big(1+O(m^{1/n})\Big)\,,
  \end{equation}
  see Remark \ref{rmk finale}. When $g$ and $\s$ are identically zero this formula was proved by Bayle and Rosales \cite{baylerosales}, and the coefficient in front of $O(m^{1/n})$ was identified by Fall in \cite{fall}, thus leading to the above mentioned criterion that, in the small volume regime, isoperimetric regions converge to boundary points of maximal mean curvature.}
\end{remark}

We now describe the proof of Theorem \ref{thm main}. An initial difficulty is excluding that minimizers break down into smaller droplets, or that they take elongated shapes with comparatively larger diameter than volume. This issue is partially addressed by a grid argument (see in particular step four in the proof of Lemma \ref{thm convergenza a K} below) where it is shown the existence of points $x_m\in A$ such that
\begin{equation}
  \label{a casa}
E_m\subset B_{x_m,C\,m^{1/2n}}\,.
\end{equation}
(Here $B_{x,r}=x+B_r$ is the ball of center $x$ and radius $r$ in $\R^n$.) Although the possibility of replacing $x_m\in A$ by $y_m\in\pa A$ with $\s(y_m)-\s_0\le C_0\,m^{1/2n}$ follows from \eqref{a casa} by  a direct variational argument, we are not able, at this stage of the proof, to improve the diameter estimate from the order $m^{1/2n}$ to the natural order $m^{1/n}$. (In other words, our droplet, rescaled by a factor $m^{1/n}$ so to bring it to unit volume, could still look like a very elongated ellipsoid.) The inclusion \eqref{a casa}, with $y_m\in\pa A$ in place of $x_m\in A$, is however sufficient to use the boundary diffeomorphisms of $A$ to map back $E_m$ from the container into our reference half-space $H$. More precisely (see Notation \ref{notation boundary of A} below) there exist positive constants $s_0$ and $r_0$ such that for every $y\in\pa A$ we can find an open set $U_y\subset\R^n$ with $B_{s_0}\subset U_y$, and a $C^{1,1}$-diffeomorphism $\phi_y:U_y\to B_{y,2r_0}$ with $\phi_y(0)=y$ and $\nabla\phi_y(0)$ an orientation preserving isometry, such that
\[
\phi_y(U_y\cap H)=B_{y,2r_0}\cap A\,,\qquad \phi_y(U_y\cap\pa H)=B_{y,2r_0}\cap\pa A\,.
\]
Thanks to \eqref{a casa}, we can thus consider the deformed and rescaled shapes
\[
F_m=\frac{\phi_{y_m}^{-1}(E_m)}{\l_m}\,,\qquad \l_m=|\phi_{y_m}^{-1}(E_m)|^{1/n}=m^{1/n}\Big(1+O(m^{1/2n})\Big)\,.
\]
Clearly $F_m\subset H$ and $|F_m|=1$, and moreover by exploiting the minimality of $E_m$ and the fact that $\s(y_m)\to\s_0$ as $m\to 0^+$, one can see that
\begin{equation}
  \label{energia intro}
  \F_{H,\s_0}(F_m)-\psi(\s_0)\le C\,m^{1/n}\,,
\end{equation}
so that $F_m$ is asymptotically optimal in \eqref{variational problem limitingx} with $\tau=\s_0$. We can thus exploit a qualitative stability theorem (see Proposition \ref{lemma psi tau}-(ii)) to show the existence of $z_m\in\pa H$ such that
\begin{equation}
  \label{L1 conv intro}
  \lim_{m\to 0}|(F_m-z_m)\Delta K(\s_0)|=0\,.
\end{equation}
The arguments described so far are contained in Lemma \ref{thm convergenza a K} below. The next step in our analysis is Lemma \ref{lemma almost minimizers}, where we show that the existence of positive constants $\Lambda$ and $\rho_0$, and of elliptic functionals $\PSI_m$ such that each $F_m$ is a $(\Lambda,\rho_0)$-minimizer of the corresponding $\PSI_m$, i.e.
\begin{equation}
  \label{almost minimizer intro}
  \PSI_m(F_m;W)\le\PSI_m(F;W)+\Lambda\,|F\Delta F_m|\,,
\end{equation}
whenever $F\subset H$, $F_m\Delta F\cc W$ for some open set $W$ with $\diam(W)<2\rho_0$; see Definition \ref{def elliptic int} and Definition \ref{def lambda min} for the terminology and notation used here. Since we can show that each $\Psi_m$ is $\l$-elliptic on a $3\rho_0$-neighborhood of $F_m$ (for some $\l\ge 1$ independent of $m$), the minimality inequality \eqref{almost minimizer intro} implies uniform volume density estimates at boundary points of each $F_m$. In turn, this information allows one to improve \eqref{L1 conv intro} into
\begin{equation}
  \label{hd conv intro}
  \lim_{m\to 0}\hd\big(F_m-z_m,K(\s_0)\big)=0\,,
\end{equation}
so that $\diam(F_m)\le C$, and thus the natural diameter estimate $\diam(E_m)\le C\,m^{1/n}$.

The next step in our analysis is to notice that if we set
\[
\Phi(\nu)=\sup\big\{x\cdot\nu:\nu\in S(\s_0)\big\}\,,\qquad\nu\in S^{n-1}\,,
\]
then, in the terminology of \cite{fonsecamuellerwulff}, $S(\s_0)$ is the Wulff shape associated to $\Phi$  and
\begin{eqnarray*}
&&\F_{H,\s_0}(F)\ge\PHI(F)=\int_{\pa^*F}\Phi(\nu_F)\,d\H^{n-1}\,,\qquad\forall F\subset H\,,
\\
&&\F_{H,\s_0}(K(\s_0))=\PHI(K(\s_0))=\psi(\s_0)\,,
\end{eqnarray*}
where $K(\s_0)$ is just a translation of the unit volume rescaling of $S(\s_0)$. In particular, by \eqref{energia intro} we have $\PHI(F_m)-\PHI(K(\s_0))\le C\,m^{1/2n}$, and then by the quantitative Wulff inequality from \cite{FigalliMaggiPratelliINVENTIONES} we infer the existence of $w_m\in\R^n$ such that
\begin{equation}
  \label{L1 conv intro quant wm}
|(F_m-w_m)\Delta K(\s_0)|\le C\,m^{1/2n}\,.
\end{equation}
We now exploit a simple geometric argument from \cite{figallindrei} together with the inclusion
\[
\frac{K(\s_0)}2\subset F_m-z_m
\]
(which follows immediately from \eqref{hd conv intro}) to conclude that one may as well take $w_m\cdot e_n=0$, and thus set $w_m=z_m\in\pa H$ in \eqref{L1 conv intro quant wm}. By combining this fact with the uniform volume density estimates for $F_m$, we are able to quantify \eqref{hd conv intro} and obtain
\begin{equation}
  \label{hd conv intro quant}
  \hd\big(F_m-z_m,K(\s_0)\big)\le C\,m^{1/2n^2}\,.
\end{equation}
In order to complete the proof of Theorem \ref{thm main} we are thus left to construct the diffeomorphisms between $M_0$ and $M_m$, and to rewrite \eqref{hd conv intro quant} in terms of $(E_m-y_m)/m^{1/n}$. We comment here only on the former task, which is achieved by combining the boundary regularity theorem from \cite{dephilippismaggiARMA} with a tool for constructing almost-normal diffeomorphsims between manifolds with boundary which was recently presented in \cite{CiLeMaIC1}. The corresponding diffeomorphisms enjoy a quite rigid structure, which should allow one to quantify more explicitly the rate of convergence in \eqref{diffeo convergence}. We leave this task for future investigations.

We now describe the organization of our paper. Section \ref{section preliminaries} is focused on the sessile droplet problem with no gravity, see \eqref{variational problem limitingx}. We first discuss some stability properties, see Proposition \ref{lemma psi tau}, and then we present an improved convergence theorem for sequence of uniform almost-minimizers of $\F_{H,\tau}$ converging in volume to $K(\tau)$, Theorem \ref{thm improved convergence to K}. In the latter result, convergence in volume is improved to $C^1$-convergence, in the sense that we extract from the almost-minimality condition the existence of $C^{1,\a}$-diffeomorphisms between the interior interfaces $M_0$ and $M_m$ converging in $C^1$ to the identity map. In section \ref{section convergence to K} we begin the proof of Theorem \ref{thm main}. In particular, in Lemma \ref{thm convergenza a K}, we obtain \eqref{a casa} (with $y_m\in\pa A$ in place of $x_m$) and prove \eqref{L1 conv intro} along the lines described above. The proof of Theorem \ref{thm main} is then concluded in section \ref{section convergence to K c1alpha}, first by proving the uniform almost-minimality of the sets $F_m$ (see Lemma \ref{lemma almost minimizers}), and then by wrapping-up the various information collected up to that point into a final discussion. We conclude this introduction by gathering the basic notation used in the paper.

\medskip

\noindent {\bf Sets in $\R^n$}: Given $x\in\R^n$ and $r>0$, the ball of center $x$ and radius $r$ is denoted by $B_{x,r}=\{y\in\R^n:|x-y|<r\}$, and we set $B_r=B_{0,r}$, $B=B_1=B_{0,1}$. Given $X,Y\subset\R^n$, the Hausdorff distance between $X$ and $Y$ is defined as
\begin{equation}
  \label{hd def}
  \hd(X,Y)=\max\big\{\sup_{x\in X}\dist(x,Y),\sup_{x\in Y}\dist(x,X)\big\}\,,
\end{equation}
while $I_\rho(X)=\{x\in\R^n:\dist(x,X)<\rho\}$ denotes the $\rho$-neighborhood of $X$, $\rho>0$.

\medskip

\noindent {\bf Manifolds in $\R^n$}: If $M\subset\R^n$ is a $k$-dimensional manifold with boundary, $1\le k\le n-1$, then we denote by $\INT(M)$ and $\bd(M)$ its interior and boundary points respectively, by $\nu_M^{co}$ the outer unit normal to $\bd(M)$ in $M$, and we set
\[
[M]_\rho=M\setminus I_\rho(\bd(M))\,,\qquad\rho>0\,.
\]

\noindent {\bf Sets of finite perimeter}: Given a Borel set $E\subset\R^n$ of locally finite perimeter in $\R^n$, we denote by $\pa^*E$ and $\nu_E$ the reduced boundary and the measure-theoretic outer unit normal of $E$. We have
\[
\ov{\pa^*E}=\Big\{x\in\R^n:0<|B_{x,r}\cap E|<|B_{x,r}|\ \forall r>0\Big\}\subset\pa E\,,
\]
and, up to modifying $E$ by a set of volume zero, one can always achieve
\[
\ov{\pa^*E}=\pa E\,,
\]
see \cite[Proposition 12.19]{maggiBOOK}. We shall always assume that the sets of finite perimeter under consideration have been modified in this way.

\medskip

\noindent {\bf Acknowledgment}: This work was supported by the NSF Grants DMS-1265910 and DMS-1361122.

\section{Some properties of droplets in half-spaces}\label{section preliminaries} In this section we discuss some basic properties of the variational problem
\begin{equation}
  \label{variational problem limiting}
  \psi(\tau)=\inf\big\{\F_{H,\tau}(F):F\subset H\,,|F|=1\big\}\,,
\end{equation}
where $H=\{x\in\R^n:x_n>0\}$ and $\tau\in(-1,1)$. Let us recall from the introduction that if we set
\begin{equation}
  \label{Stau}
  S(\tau)=\{x\in B:x_n>-\tau\}\,,
\end{equation}
then the unique minimizer $K(\tau)$ in \eqref{variational problem limiting} with horizontal barycenter at the origin is given by the formula
\begin{equation}
  \label{Ktau from Stau}
  K(\tau)=\tau\,e_n+\frac{S(\tau)}{|S(\tau)|^{1/n}}\,.
\end{equation}
In other words, $F$ is a minimizer in \eqref{variational problem limiting} if and only if $F=z+K(\tau)$ for some $z\in\pa H$. In section \ref{section sessile} we discuss some stability properties of \eqref{variational problem limiting}, while section \ref{section improved convergence} contains an improved convergence theorem towards the ideal droplet $K(\tau)$.

\subsection{Stability properties}\label{section sessile}  The following proposition collects the properties of problem \eqref{variational problem limiting} that we shall need in the study of \eqref{variational problem m}. Property (i) consists just in the monotonicity of $\psi$, while property (ii) is a qualitative stability statement. In property (iii) we exploit the main result of \cite{FigalliMaggiPratelliINVENTIONES} to quantify stability under a technical containment assumption. This containment assumption is not needed, and indeed it could be eliminated by mimicking the arguments in \cite{figallindrei}. However, this more general result is not needed here.

\begin{proposition}\label{lemma psi tau}
  (i) One has $\psi'(\tau)>0$ for every $\tau\in(-1,1)$.

  \medskip

  \noindent (ii) If $\{F_h\}_{h\in\N}$ is a sequence of subsets of $H$ with $|F_h|=1$ for every $h\in\N$ and
  \[
  \lim_{h\to\infty}\F_{H,\tau}(F_h)=\psi(\tau)\,,
  \]
  then there exists $\{z_h\}_{h\in\N}\subset\pa H$ such that, up to extracting subsequences, one has
  \[
  \lim_{h\to\infty}|(F_h+z_h)\Delta K(\tau)|=0\,.
  \]

  \medskip

  \noindent (iii) There exist positive constants $\e(n,\tau)$ and $c(n,\tau)$ with the following property: If $F\subset H$ with $|F|=1$ and
  \begin{equation}
    \label{hp technical}
  \frac{K(\tau)}2\subset F\,,\qquad \F_{H,\tau}(F)\le (1+\e(n,\tau))\,\psi(\tau)\,,
  \end{equation}
  then
  \begin{equation}
    \label{quantitative sessile problem}
      \F_{H,\tau}(F)-\psi(\tau)\ge c(n,\tau)\,\inf_{z\in\pa H}\,|(F-z)\Delta K(\tau)|^2\,.
  \end{equation}
\end{proposition}

\begin{proof}[Proof of Proposition \ref{lemma psi tau}-(i)] Given $\tau\in(-1,1)$, let
\[
V(\tau)=|S(\tau)|\,,\qquad A(\tau)=P(S(\tau);\{x_n>-\tau\})\,,\qquad A_0(\tau)=P(S(\tau);\{x_n=-\tau\})\,.
\]
Since $\nu_B\cdot(-e_n)=\tau$ along $\{x_n=-\tau\}\cap\pa S(\tau)$, by \eqref{Ktau from Stau} we find that
\[
\psi(\tau)=\F_{H,\tau}(K(\tau))=\frac{A(\tau)+\tau\,A_0(\tau)}{V(\tau)^{(n-1)/n}}\,.
\]
We now notice that, if $\o_k$ denotes the volume of the unit sphere in $\R^k$, then
\[
V(\t)=\int_{-\t}^{1}\o_{n-1} (1-\rho^2)^{(n-1)/2}d\rho\,,\quad A(\t)=\int_{-\t}^{1}(n-1)\o_{n-1} (1-\rho^2)^{(n-3)/2}d\rho\,,
\]
while $A_0(\tau)=\o_{n-1}(1-\t^2)^{(n-1)/2}$. On noticing that
\[
V'=A_0\,,\qquad (A+\tau A_0)'=(n-1)\frac{A_0}{1-\tau^2}+A_0-(n-1)\frac{\tau^2\,A_0}{1-\tau^2}=n\,A_0\,,
\]
we find that
\[
\psi'=\frac{1}{V^{2-(2/n)}}\Big(V^{1-(1/n)}n\,A_0-\frac{n-1}n\,V^{-1/n}\,A_0\,(A+\tau A_0)\Big)=\frac{A_0}{n\,V^{2-(1/n)}}\,\vphi\,,
\]
where $\vphi=n^2\,V-(n-1)(A+\tau\,A_0)$. By the divergence theorem, $A(0)=n\,V(0)$, where $V(0)=\om_n/2$, so that $\vphi(0)=n\,\om_n/2$. Moreover, $\vphi'=n^2A_0-(n-1)nA_0=nA_0$, thus
\[
\vphi(\tau)=n\Big(\frac{\om_n}2+\int_0^\tau A_0\Big)=n\Big(\frac{\om_n}2+{\rm sign}(\tau)\,\Big|\Big\{x\in B:0<x_n<|\tau|\Big\}\Big|\Big)>0\,,
\]
for every $\tau\in(-1,1)$. This proves that $\psi'>0$ on $(-1,1)$.
\end{proof}

We now discuss statement (ii). As usual the issue is ensuring compactness in volume. We solve this problem by combining slicing with isoperimetry to prove that one can always reduce to consider sequences of sets $F_h$ with uniformly bounded diameters, which therefore are compact in volume modulo horizontal translations. The main modifications with respect to the case of the standard isoperimetric problem, corresponding formally to $\tau=1$, (see, for example, \cite[Lemma 5.1]{fuscomaggipratelli}, which in turn was inspired by \cite[Theorem 3.1]{fonsecamuellerwulff}) are found in the case $\tau<0$. Before entering into the proof, it is convenient to introduce some notation and terminology in analogy with \cite{fuscomaggipratelli}. Given $F\subset H$ and $\tau\in(-1,1)$ we define the {\it deficit} of $F$ (relatively to the variational problem \eqref{variational problem limiting}) as
\[
\de_\tau(F)=\frac{\F_{H,\tau}(F)}{\psi(\tau)\,|F|^{(n-1)/n}}-1\,.
\]
For every $\l>0$ we have $\de_\tau(\l\,F)=\de_\tau(F)\ge0$. Moreover, by \cite[Theorem 19.21]{maggiBOOK}, we have that $\de_\tau(F)=0$ if and only if $|F\Delta(z+r\,K(\tau))|=0$ for some $z\in\pa H$ and $r>0$. Correspondingly we define the {\it asymmetry index} of $F$ (again, relatively to problem \eqref{variational problem limiting}) as
\[
\a_\tau(F)=\inf\Big\{\frac{|F\Delta(z+r\,K(\tau))|}{|F|}:r^n=|F|\,,z\in\pa H\Big\}\,.
\]
With this terminology in force, statement (ii) is equivalent in saying that if $\de_\tau(F_h)\to 0$, then $\a_\tau(F_h)\to 0$. The key point in the proof will thus be obtaining \eqref{bounded} and \eqref{bounded vertical} below.

\begin{proof}[Proof of Proposition \ref{lemma psi tau}-(ii)] \noindent {\it Step one}: As a preparatory remark, we show that if $G, F\subset\R^n$ with
\begin{equation}
  \label{checkGF}
  |G|\le|F|\,,\qquad |F\Delta G|\le \e\,|F|,\qquad \mathcal{F}_{H,\t}(G)\le\mathcal{F}_{H,\t}(F)+\,(\e \,|F|)^{(n-1)/n}\,,
\end{equation}
for some $\e\in(0,1)$, then
\begin{equation}
  \label{continuity alpha and delta}
  |\a_\tau(F)-\a_\tau(G)|\le 3\,\e,\qquad \delta_\tau(G)\le\frac{\delta_\tau(F)+(1+\psi(\tau)^{-1})\,\e^{(n-1)/n}}{(1-\e)^{(n-1)/n}}\,\,.
\end{equation}
Since the volume of the symmetric difference defines a distance on subsets of $\R^n$, we easily find that $||F|\a_\tau(F)-|G|\a_\tau(G)|\le|F\Delta G|$. Hence, by $\a_\tau(G)<2$,
\[
|F||\a_\tau(F)-\a_\tau(G)|\le |F\Delta G|+||F|-|G||\,\a_\tau(G)\le(1+\a_\tau(G))\,|F\Delta G|\le 3\,|F\Delta G|\,,
\]
and the first estimate in \eqref{continuity alpha and delta} follows. Next we notice that
\begin{eqnarray*}
  |G|^{(n-1)/n}\de_\tau(G)-|F|^{(n-1)/n}\de_\tau(F)&\le&\frac{\F_\tau(G)-\F_\tau(F)}{\psi(\tau)}+|F|^{(n-1)/n}-|G|^{(n-1)/n}
  \\
  &\le&\frac{(\e|F|)^{(n-1)/n}}{\psi(\tau)}+||F|-|G||^{(n-1)/n}
  \\
  &\le&\Big(1+\frac1{\psi(\tau)}\Big)\,(\e|F|)^{(n-1)/n}\,,
\end{eqnarray*}
and we conclude the proof by $|G|\ge (1-\e)|F|$.

\bigskip

\noindent {\it Step two}: We claim that if $F\subset H$ with $\de_\tau(F)\le\de_0(n,\tau)<1$, then there exists $G\subset H$ with
\begin{equation}
  \label{bounded}
  |\a_\tau(G)-\a_\tau(F)|\le C(n,\tau)\,\de_\tau(F)\,,\quad \de_\tau(G)\le C(n,\tau)\,\de_\t(F)\,,\quad \sup_{x,y\in G}\frac{|x_1-y_1|}{|F|^{1/n}}\le C(n,\tau)\,,
\end{equation}
where $\de_0(n,\tau)$ and $C(n,\tau)$ are suitable positive constants. Without loss of generality, we may assume that $|F|=1$ and $F=F^{(1)}$, the set of points of density $1$ of $F$. Thus it suffices to prove
\begin{equation}
  \label{bounded x1}
  \sup_{x,y\in G}|x_1-y_1|\le C(n,\tau)\,,
\end{equation}
together with
\begin{equation}
  \label{bounded x}
 |F\Delta G| \le C(n,\tau)\,\de_\tau(F)^{n/(n-1)}\,,\qquad \F_{H,\tau}(G)\le \F_{H,\tau}(F)+C(n,\tau)\,\de_\tau(F)\,,
\end{equation}
which take the place of the first two inequalities in \eqref{bounded} by step one. Let us set
\begin{equation}
  \label{v and p}
  v(t)=|F\cap\{x_1 <t\}|\,,\qquad s(t)=\H^{n-1}(F\cap\{x_1 =t\})\,,
\end{equation}
so that $v(t)$ is absolutely continuous on $\R$ with $v'(t)=s(t)$ for a.e. $t\in\R$ thanks to Fubini's theorem. We notice that for every $t\in\R$ one has
\begin{equation}
  \label{divergence theorem application}
  \begin{split}
s(t)=\H^{n-1}(F\cap\{x_1 =t\})\le P(F;H\cap\{x_1 <t\})\,,
\\
P(F;\{x_1 <t\}\cap\pa H)\le P(F;H\cap\{x_1 <t\})\,.
\end{split}
\end{equation}
Indeed, by \cite[Equation (16.7)]{maggiBOOK} and $F=F^{(1)}$ one finds that
\begin{equation}
  \label{normale Ftmeno}
  \begin{split}
  \pa^*(F\cap \{x_1 <t\})=&\big(H\cap\{x_1 <t\}\cap\pa^*F\big)
  \\
  &\cup\big(\{x_1 <t\}\cap\pa H\cap\pa^*F\big)
  \\
  &\cup\big(F\cap\{x_1 =t\}\big)
  \\
  &\cup\big(\pa^*F\cap\{x_1=t\}\cap\{\nu_F=e_1\}\big)
\end{split}
\end{equation}
where the identity holds up to $\H^{n-1}$-negligible sets, and where the elements of the union are disjoint. Since the outer unit normal to $F\cap \{x_1 <t\}$ coincides with $\nu_F$ on the first set on the right-hand side, with $-e_n$ on the second one, and with $e_1$ on the third and the fourth one, the first inequality in \eqref{divergence theorem application} follows by applying the divergence theorem on $F\cap\{x_1 <t\}$ to the constant vector field $f(x)=e_1$,
\begin{eqnarray*}
0&=&\int_{H\cap\{x_1 <t\}\cap\pa^*F}e_1\cdot\nu_F+\H^{n-1}(\{x_1 =t\}\cap F)+\H^{n-1}(\{x_1 =t\}\cap \pa^*F\cup\{\nu_F=e_1\})
\\
&\ge&\int_{H\cap\{x_1 <t\}\cap\pa^*F}e_1\cdot\nu_F+\H^{n-1}(\{x_1 =t\}\cap F)\,,
\end{eqnarray*}
while the second inequality follows by using the vector field $f(x)=e_n$,
\[
0=\int_{H\cap\{x_1 <t\}\cap\pa^*F}e_n\cdot\nu_F-P(F;\{x_1 <t\}\cap\pa H)\,.
\]
The actual proof requires a truncation of the vector field $f(x)$. We omit the details and refer to \cite[Proposition 19.22]{maggiBOOK} for a complete exposition of an identical argument.

With \eqref{divergence theorem application} at hand, we now prove the existence of $G\subset F$ such that \eqref{bounded x} and \eqref{bounded x1} hold. For $t\in\R$ we set
\[
F^-_t=F\cap\{x_1<t\}\,,\qquad F^+_t=F\cap\{x_1>t\}\,.
\]
By \eqref{normale Ftmeno}, and by an analogous formula for $F_t^+$, one has
\[
\F_{H,\tau}(F_t^+)+\F_{H,\tau}(F_t^-)\le\F_{H,\tau}(F)+2\,s(t)\,,\qquad \forall t\in\R\,.
\]
(The inequality sign depends on the possibility that $\H^{n-1}(\pa^*F\cap\{x_1=t\})>0$.) By definition of $\psi(\tau)$ we have $\F_{H,\tau}(F_t^\pm)\ge\psi(\tau)|F_t^\pm|^{(n-1)/n}$, so that
\begin{eqnarray*}
2\,s(t)+\psi(\tau)\,\de_\tau(F)&=&2\,s(t)+\F_{H,\tau}(F)-\psi(\tau)\ge\F_{H,\tau}(F_t^+)+\F_{H,\tau}(F_t^-)-\psi(\tau)
\\
&\ge&\psi(\tau)\,\Psi(v(t))\,,\qquad\forall t\in\R\,,
\end{eqnarray*}
where  $\Psi(\g)=(1-\g)^{(n-1)/n}+\g^{(n-1)/n}-1$, $\g\in[0,1]$. We rearrange this inequality as
\begin{equation}
  \label{alps1}
  s(t)\ge \frac{\psi(\tau)}2\,\Big(\Psi(v(t))-\de_\tau(F)\Big)\,,\qquad\forall t\in\R\,.
\end{equation}
Let us   notice that we can find $\k(n)>0$ such that
\begin{equation}
  \label{alps2}
  \Psi(\g)\ge\k(n)\,\min\{\g,1-\g\}^{(n-1)/n}\,,\qquad\forall \g\in[0,1]\,.
\end{equation}
We now use the assumption that $\de_\t(F)\le\de_0(n,\tau)$ to ensure that, if we set
\begin{equation}
  \label{def t1}
  t_1=\inf\Big\{t\in\R:\frac{\k(n)\,v(t)^{(n-1)/n}}2\ge\de_\t(F)\Big\}\,,
\end{equation}
then $t_1\in\R$. Since $v$ is increasing, we have $\k(n)\,v(t)^{(n-1)/n}\ge 2\de_\t(F)$ for every $t>t_1$, and thus we can apply \eqref{alps1} and \eqref{alps2} to find that if $t>t_1$ with $v(t)\le 1/2$, then
\begin{equation}
  \label{roads 1}
  s(t)\ge \frac{\psi(\tau)}2\,\Big(\k(n)\,v(t)^{(n-1)/n}-\de_\tau(F)\Big)\ge\frac{\psi(\tau)\k(n)}4\,\,v(t)^{(n-1)/n}\,;
\end{equation}
similarly, if we define $t_2$ by
\begin{equation}
  \label{def t2}
t_2=\sup\Big\{t\in\R:\frac{\k(n)\,(1-v(t))^{(n-1)/n}}2\ge\de_\t(F)\Big\}\,,
\end{equation}
then $t_2\in\R$ with $\k(n)\,(1-v(t))^{(n-1)/n}\ge 2\de_\tau(F)$ for every $t<t_2$, and again by \eqref{alps1} and \eqref{alps2}
\begin{equation}
  \label{roads 2}
  s(t)\ge \frac{\psi(\tau)}2\,\Big(\k(n)\,(1-v(t))^{(n-1)/n}-\de_\tau(F)\Big)\ge\frac{\psi(\tau)\k(n)}4\,\,(1-v(t))^{(n-1)/n}\,,
\end{equation}
whenever $t<t_2$ with $v(t)\ge 1/2$. In conclusion,
\begin{equation}
  \label{for ae}
  s(t)\ge\frac{\psi(\tau)\k(n)}4\,\,\min\{v(t)\,,1-v(t)\}^{(n-1)/n}\,,\qquad\forall\ t\in(t_1,t_2)\,,
\end{equation}
so that, by taking $s=v'$ a.e. on $\R$ into account,
\[
\frac{\psi(\tau)\,\k(n)(t_2-t_1)}{4}\le \int_{t_1}^{t_2}\frac{v'}{\min\{v,1-v\}^{(n-1)/n}}\le\int_0^1\frac{d\g}{\min\{\g,1-\g\}^{(n-1)/n}}=C(n)\,,
\]
that is $t_2-t_1\le C(n,\tau)$. Hence, the set
\[
G'=F\cap\{t_1<x_1<t_2\}
\]
has directional diameter along the $x_1$-axis bounded by $C(n,\tau)$, with
\begin{equation}
  \label{F delta G}
  |F\setminus G'|=v(t_1)+(1-v(t_2))\le 2\,\Big(\frac{2\,\de_\tau(F)}{\k(n)}\Big)^{n/(n-1)}\,.
\end{equation}
We now split the argument depending on the sign of $\tau$.

When $\tau\ge0$ we conclude the proof of \eqref{bounded x} and \eqref{bounded x1} by setting $G=G'$. Indeed, with this choice, \eqref{bounded x1} is immediate from $t_2-t_1\le C(n,\tau)$, while the first bound in \eqref{bounded x} follows from \eqref{F delta G}. Moreover, by \eqref{divergence theorem application},
\begin{eqnarray}\label{rearrange them}
\F_{H,\tau}(F)-\F_{H,\tau}(G)
&\ge&P(F;H\cap\{x_1<t_1\})-s(t_1)+\tau\,P(F;\pa H\cap\{x_1<t_1\})
\\\nonumber
&&+P(F;H\cap\{x_1>t_2\})-s(t_2)+\tau\,P(F;\pa H\cap\{x_2>t_1\})\,,
\end{eqnarray}
where $P(F;H\cap\{x_1<t_1\})\ge s(t_1)$  by \eqref{divergence theorem application}, and similarly $P(F;H\cap\{x_1>t_2\})\ge s(t_2)$ (and where the first inequality sign depends on the possibility that $\H^{n-1}(\pa^*F\cap\{x_1=t_i\})>0$ for $i=1,2$.) Thus $\tau\ge0$ implies $\F_{H,\tau}(F)\ge\F_{H,\tau}(G)$, and the second bound in \eqref{bounded x}.

When $\tau<0$ we can fix the argument by cutting $F$ at nearby levels of $t$ such that $s(t)$ is sufficiently small in terms of deficit. To make this precise, let us consider the sets
\begin{eqnarray*}
I_1=\Big\{t\in\R:v(t)\le \frac12\,,\quad s(t)<\frac{\psi(\tau)\de_\tau(F)}2\Big\}\,,
\\
I_2=\Big\{t\in\R:v(t)\ge \frac12\,,\quad s(t)<\frac{\psi(\tau)\de_\tau(F)}2\Big\}\,,
\end{eqnarray*}
and define
\[
t_1^*=\left\{
\begin{array}{l l}
  \inf\{t\in\R:v(t)>0\}\,,&\mbox{if $I_1=\emptyset$}
  \\
  \sup\,I_1\,,&\mbox{if $I_1\ne\emptyset$}
\end{array}\right .
\qquad
t_2^*=\left\{
\begin{array}{l l}
  \sup\{t\in\R:v(t)>0\}\,,&\mbox{if $I_2=\emptyset$}
  \\
  \inf\,I_2\,,&\mbox{if $I_2\ne\emptyset$}
\end{array}
\right .
\,.
\]
We first note that
\begin{eqnarray}
\diam(\{0<v<1/2\})\le C(n,\tau)\,,&&\qquad\mbox{if $I_1=\emptyset$}\,,
\\
\diam(\{1/2< v<1\})\le C(n,\tau)\,,&&\qquad\mbox{if $I_2=\emptyset$}\,;
\end{eqnarray}
indeed, if for example $I_1=\emptyset$, then by $s=v'$ a.e. on $\R$ we find
\[
\frac12\ge\int_{\{0<v< 1/2\}}s(t)\,dt\ge \diam(\{0<v< 1/2\})\,\frac{\psi(\tau)\de_\tau(F)}2\,.
\]
Next we remark that if $I_1\ne\emptyset$, then $t_1^*\in\R$ and actually $t_1^*\le t_1$: indeed, by the same argument leading to \eqref{roads 1} one can deduce that
\[
s(t)\ge\frac{\psi(\tau)}2\,\de_\t(F)\,,\qquad\mbox{for every $t>t_1$ with} \ v(t)\le\frac12\,.
\]
Similarly, by arguing as in the proof of \eqref{roads 2}, we see that if $I_2\ne\emptyset$, then $t_2^*\in\R$ with $t_2^*\ge t_2$, as
\[
s(t)\ge\frac{\psi(\tau)}2\,\de_\t(F)\,,\qquad\mbox{for every $t<t_2$ with} \ v(t)\ge\frac12\,.
\]
Finally, we notice that if $I_1\ne\emptyset$ or $I_2\ne\emptyset$, then one has, respectively,
\[
t_1-t_1^*\le C(n,\tau)\,,\qquad t_2^*-t_2\le C(n,\tau)\,.
\]
To prove the first relation, notice that for every $t\in (t_1^*,t_1)$ one has $\psi(\tau)\de_\tau(F)\le 2\,s(t)$, and thus $s=v'$ a.e. on $\R$ and $v(t_1)=(2\de_\t(F)/\k(n))^{n/(n-1)}$ give
\[
t_1-t_1^*\le\frac2{\psi(\tau)\de_\tau(F)}\int_{t_1^*}^{t_1} v'\le
\frac{2\,v(t_1)}{\psi(\tau)\de_\tau(F)}\le C(n,\tau)\,\de_\tau(F)^{1/(n-1)}\,;
\]
similarly, for every $t\in(t_2,t_2^*)$ one has $\psi(\tau)\de_\tau(F)\le 2\,s(t)$, and thus $s=v'$ a.e. on $\R$ and $v(t_2)\ge 1-(2\de_\t(F)/\k(n))^{n/(n-1)}$ give
\[
t_2^*-t_2\le
\frac{2\,(v(t_2^*)-v(t_2))}{\psi(\tau)\de_\tau(F)}
=
\frac{2\,[(1-v(t_2))-(1-v(t_2^*))]}{\psi(\tau)\de_\tau(F)}
\le C(n,\tau)\,\de_\tau(F)^{1/(n-1)}\,;
\]
With these remarks at hand we finally set
\[
t_1^{**}
\]
to be $\inf\{v>0\}-1$ if $I_1=\emptyset$ or to be any $t<t_1^*$ with
\[
s(t_1^{**})<\frac{\psi(\tau)\de_\tau(F)}2\,,\qquad t_1^*-t_1^{**}\le 1\,,
\]
in case $I_1\ne\emptyset$; we set
\[
t_2^{**}
\]
to be $\sup\{v>0\}+1$ if $I_2=\emptyset$ or to be any $t>t_2^*$ with
\[
s(t_2^{**})<\frac{\psi(\tau)\de_\tau(F)}2\,,\qquad t_2^{**}-t_2^{*}\le 1\,,
\]
in case $I_2\ne\emptyset$; and finally define $G$ by taking
\[
G=F\cap\{t_1^{**}<x_1<t_2^{**}\}\,.
\]
Notice that the above remarks show that it must be
\[
t_2^{**}-t_1^{**}\le C(n,\tau)\,,\qquad\max\Big\{s(t_1^{**}),s(t_2^{**})\Big\}\le \frac{\psi(\tau)\de_\tau(F)}2\,,
\]
so that \eqref{bounded x1} holds. Since $G'\subset G=F\cap\{t_1<x_1<t_2\}$ we have
\[
|F\setminus G|\le|F\setminus G'|\le C(n,\tau)\,\de_\tau(F)^{n/(n-1)}\,,
\]
and thus the first bound in \eqref{bounded x} holds. At the same time if we write down \eqref{rearrange them} with $t_i^{**}$ in place of $t_i$ and exploit the inequalities $|\tau|\,P(F;\pa H\cap\{x_1<t_i^{**}\})<P(F;H\cap\{x_1<t_i\})$ (recall the second inequality in \eqref{divergence theorem application}), then we find
\[
\F_{H,\tau}(F)-\F_{H,\tau}(G)\ge -s(t_1^{**})-s(t_2^{**})\ge -\psi(\tau)\,\de_\tau(F)\,,
\]
which is the second bound in \eqref{bounded x}. This completes step two.

\bigskip

\noindent {\it Step three}: We show that if $F\subset H$ with $\de_\tau(F)\le\de_0(n,\tau)$, then there exists $G\subset H$ with
\begin{equation}
  \label{bounded vertical}
  |\a_\tau(G)-\a_\tau(F)|\le C(n,\tau)\,\de_\tau(F)\,,\quad \de_\tau(G)\le C(n,\tau)\,\de_\t(F)\,,\quad \sup_{x\in G}\frac{x_n}{|F|^{1/n}}\le C(n,\tau)\,.
\end{equation}
Once again, we can prove this with $|F|=1$ and $F=F^{(1)}$. Similarly to step two we set
\[
v(t)=|F\cap\{x_n<t\}|\,,\quad s(t)=\H^{n-1}(F\cap\{x_n=t\})\,,\quad p(t)=\H^{n-2}(\pa^*F\cap\{x_n=t\})\,,
\]
and notice that, thanks to $F=F^{(1)}$ and \cite[Equation (16.7)]{maggiBOOK}, for every $t\in\R$ one has
\[
\F_{H,\tau}(F_t^+)+\F_{H,\tau}(F_t^-)\le\F_{H,\tau}(F)+2\,s(t)\,,
\]
where now $F_t^+=F\cap\{x_n>t\}$ and $F_t^-=F\cap\{x_n<t\}$, and where the inequality sign holds since it may happen that $\H^{n-1}(\pa^*F\cap\{x_n=t\})>0$ for some (but at most for countably many) values of $t$. By setting, similarly to what done in \eqref{def t1} and \eqref{def t2},
\[
t_1=\inf\Big\{t>0:\frac{\k(n)\,v(t)^{(n-1)/n}}2\ge\de_\t(F)\Big\}\,,\ \
t_2=\sup\Big\{t>0:\frac{\k(n)\,(1-v(t))^{(n-1)/n}}2\ge\de_\t(F)\Big\}\,,
\]
and by repeating the same arguments we find once again $t_2-t_1\le C(n,\tau)$. The step will be then be completed, once again thanks to step one, by setting $G=F\cap\{x_1<t_2\}$ and by showing that $t_1\le C(n,\tau)$. We shall actually prove that $t_1\le C(n,\tau)\,\de_\tau(F)^{n/(n-1)}$. To this end, we first notice that for a.e. $t\in(0,t_1)$ and by definition of $\psi(1)$,
\begin{eqnarray*}
-\tau\,P(F;\pa H)+\psi(\tau)\,\de_\t(F)&=&P(F;H)-\psi(\tau)
\ge P(F;\{x_n>t\})-\psi(\tau)
\\
&=&P(F\cap\{x_n>t\})-\psi(\tau)-s(t)
\\
&\ge&\psi(1)\,|F\cap\{x_n>t\}|^{(n-1)/n}-\psi(\tau)-s(t)
\\
&\ge&\psi(1)-\psi(\tau)-C(n,\tau)\,v(t_1)-s(t)\,.
\end{eqnarray*}
If we integrate this inequality on $(0,t_1)$ and take into account that $\int_0^{t_1}s(t)\,dt=v(t_1)\le C(n,\tau)\,\de_\tau(F)^{n/(n-1)}$, then we find, provided $\de_0(n,\tau)$ is small enough,
\begin{eqnarray*}
-\tau\,P(F;\pa H)\,t_1&\ge& (\psi(1)-\psi(\tau))\,t_1-C(n,\tau)\,\de_\t(F)\,t_1-C(n,\tau)\,\de_\tau(F)^{n/(n-1)}
\\
&\ge& \frac{\psi(1)-\psi(\tau)}2\,t_1-C(n,\tau)\,\de_\tau(F)^{n/(n-1)}\,.
\end{eqnarray*}
When $\tau\ge0$, then the left-hand side of this last estimate is negative, and thus we obtain $t_1\le C(n,\tau)\,\,\de_\tau(F)^{n/(n-1)}$, as claimed. Assuming from now on that $\tau\le0$, we rewrite the above estimate as
\begin{equation}
  \label{lower bound s0}
  P(F;\pa H)\ge c(n,\tau)-C(n,\tau)\,\frac{\de_\tau(F)^{n/(n-1)}}{t_1}\,.
\end{equation}
Now, by exploiting the divergence theorem we easily see that, up to possibly excluding countably many values of $t$, $s(t)\to P(F;\pa H)$ as $t\to 0^+$. By combining this fact with \cite[Equation 19.57]{maggiBOOK} we see that
\[
P(F;\pa H)\le P(F\cap\{x_n<t\};H)=P(F;\{0<x_n<t\})+s(t)\,,\qquad\mbox{for a.e. $t>0$}\,,
\]
and thus $\tau\,P(F;\pa H)\ge \tau\,P(F:\{0<x_n<t\})+\tau\,s(t)$. In particular, for a.e. $t\in(0,t_1)$,
\begin{eqnarray}\nonumber
\psi(\tau)\,\de_\tau(F)
&=&P(F;\{x_n>t\})+P(F;\{t>x_n>0\})+\tau\,s(0)-\psi(\tau)
\\\nonumber
&\ge&P(F;\{x_n>t\})+(1+\tau)\,P(F;\{t>x_n>0\})+\tau\,s(t)-\psi(\tau)
\\\nonumber
&=&\F_{H,\tau}(F\cap\{x_n>t\}-t\,e_n)-\psi(\tau)+(1+\tau)\,P(F;\{t>x_n>0\})
\\\nonumber
&\ge&\psi(\tau)\,|F\cap\{x_n>t\}|^{(n-1)/n}-\psi(\tau)
+(1+\tau)\,P(F;\{t>x_n>0\})
\\\nonumber
&\ge&\psi(\tau)\,(1-v(t_1))^{(n-1)/n}-\psi(\tau)
+(1+\tau)\,P(F;\{t>x_n>0\})
\\\label{lb fine}
&\ge&-C(n,\tau)\,\de_\tau(F)^{n/(n-1)}+(1+\tau)\,P(F;\{t>x_n>0\})\,.
\end{eqnarray}
Now, by a minor modification of \cite[Theorem 18.11]{maggiBOOK}, we find that
\begin{eqnarray*}
P(F;\{t>x_n>0\})\ge\int_0^t\sqrt{s'(t)^2+p(t)^2}\,dt\ge\int_0^t|s'(t)|\,dt\ge P(F;\pa H)-s(t)\,,
\end{eqnarray*}
(where we have exploited again $s(t)\to P(F;\pa H)$ as $t\to 0^+$) so that, integrating over $(0,t_1)$ and taking \eqref{lower bound s0} into account
\[
\int_0^{t_1}P(F;\{t>x_n>0\})\,dt\ge t_1\,P(F;\pa H)-v(t_1)\ge c(n,\tau)\,t_1-C(n,\tau)\,\de_\tau(F)^{n/(n-1)}\,.
\]
By first integrating \eqref{lb fine} over $(0,t_1)$ and by then plugging this last inequality, we find that
\[
C(n,\tau)\,t_1\,\de_\t(F)\ge (1+\tau) \Big(c(n,\tau)\,t_1-C(n,\tau)\,\de_\tau(F)^{n/(n-1)}\Big)\,,
\]
which, for $\de_0(n,\tau)$ small enough, implies $t_1\le C(n,\tau)\,\de_\tau(F)^{n/(n-1)}$.

\bigskip

\noindent {\it Step four}: We finally prove the statement. Arguing by contradiction, we assume the existence of $\eta>0$ and of a sequence of sets $F_h\subset H$ such that $\delta_\t(F_h)\to 0$ but $\a_\t(F_h)\ge\eta$. Without loss of generality, we can assume that $|F_h|=1$. By step two and step three we can find sets $G_h\subset H$ such that $\diam(G_h)\le C(n,\tau)$, $\sup_{x\in G_h}|x_n|\le C(n,\tau)$, $|F_h\Delta G_h|\to 0$, $\F_{H,\tau}(G_h)\to\psi(\tau)$ and $\a_\tau(G_h)\ge\eta/2$. By \cite[Equation (19.58)]{maggiBOOK} we have $\F_{H,\tau}(G_h)\ge (1+\tau)\,P(G_h)/2$, so that $\F_{H,\tau}(G_h)\to\psi(\tau)$ implies $P(G_h)\le C(n,\tau)$. This bound, together with $\diam(G_h)\le C(n,\tau)$ and $\sup_{x\in G_h}|x_n|\le C(n,\tau)$, implies that up to extracting subsequences and up to horizontal translations, $|G_h\Delta G|\to 0$ for some $G\subset H$. Clearly $\a_\tau(G)\ge\eta/2$, however, by lower semicontinuity of $\F_{h,\tau}$ (see \cite[Proposition 19.27]{maggiBOOK}) it must be
\[
\F_{H,\tau}(G)\le\liminf_{h\to\infty}\F_{H,\tau}(G_h)=\psi(\tau)\,,
\]
where $\F_{H,\tau}(G)\ge\psi(\tau)$ as $|F_h\Delta G_h|\to 0$ implies $|G|=1$. In conclusion, $\F_{H,\tau}(G)=\psi(\tau)$ with $|G|=1$, and thus $G=z+K(\tau)$ for some $z\in\pa H$. But then $\a_\t(G)=0$, a contradiction.
\end{proof}

We conclude this section by showing the validity of property (iii) in Proposition \ref{lemma psi tau}. To this end, it is convenient to define $\Phi:S^{n-1}\to(0,\infty)$ by setting
\[
\Phi(\nu)=\sup\Big\{x\cdot\nu:x\in S(\tau)\Big\}=\sup\Big\{x\cdot\nu:|x|<1\,,x_n>-\tau\Big\}\,,\qquad\nu\in S^{n-1}\,,
\]
and then consider a corresponding anisotropic perimeter functional $\PHI$ defined by setting
\[
\PHI(F)=\int_{\pa^*F}\Phi(\nu_F(x))\,d\H^{n-1}_x\in[0,\infty]\,,
\]
whenever $F$ is of locally finite perimeter in $\R^n$. It is immediate to check that
\begin{equation}
  \label{calibrazione phi}
  \Phi(-e_n)=-\tau\,,\qquad \mbox{$\Phi(\nu)=1$ if $\nu\cdot e_n>-\tau$}\,,
\end{equation}
so that
\begin{equation}
  \label{calibrazione phi K}
  \PHI(K(\tau))=P(K(\tau);H)+\tau\,P(K(\tau);\pa H)=\psi(\tau)\,,
\end{equation}
while
\begin{equation}
  \label{calibrazione phi F}
  \PHI(F)=\int_{H\cap\pa^*F}\Phi(\nu_F)\,d\H^{n-1}+\tau\,P(F;\pa H)\le\F_{H,\tau}(F)\,,\qquad\forall F\subset H\,,
\end{equation}
where of course we have used that $F\subset H$ implies $\nu_F=-e_n$ for $\H^{n-1}$-a.e. $x\in\pa^*F\cap\pa H$ (see, e.g., \cite[Exercise 16.6]{maggiBOOK}) as well as that $\Phi\le 1$ on $S^{n-1}$.

\begin{proof}
  [Proof of Proposition \ref{lemma psi tau}-(iii)] By \eqref{calibrazione phi K}, \eqref{calibrazione phi F} and the main result in \cite{FigalliMaggiPratelliINVENTIONES} we have that if $F\subset H$ with $|F|=1$, then for some $w\in\R^n$
  \begin{equation}
    \label{fimp}
      \F_{H,\tau}(F)-\psi(\tau) \ge\PHI(F)-\PHI(K(\tau))\ge c_0(n)\,\PHI(K(\tau))\,|F\Delta (w+K(\tau))|^2\,,
  \end{equation}
  for some positive constant $c_0(n)$. Of course $w=z+t\,e_n$ where $z\in\pa H$ and $t\in\R$. If $t<0$, then by $F\subset H$ we obtain
  \begin{eqnarray}\nonumber
  |F\Delta (w+K(\tau))|&\ge&|(w+K(\tau))\setminus F|\ge |(w+K(\tau))\setminus H|=|(t\,e_n+K(\tau))\setminus H|
  \\\label{fimp2}
  &=&|K(\tau)\cap\{0<x_n<-t\}|\,.
  \end{eqnarray}
  By combining \eqref{fimp}, \eqref{fimp2}, and \eqref{hp technical} we find that
  \[
  \e(n,\tau)\ge c_0(n)\,|K(\tau)\cap\{0<x_n<-t\}|^2\,,
  \]
  so that if $\e(n,\tau)/c_0(n)\le (|K(\tau)|/2)^2$ then $|t|\le t_0(n,\tau)<1$ and thus $|K(\tau)\cap\{0<x_n<-t\}|\ge c(n,\tau)\,|t|$; by combining this last inequality with \eqref{fimp2} and \eqref{fimp} we thus conclude that if $t<0$, then
  \begin{equation}
    \label{calibrazione t}
    \F_{H,\tau}(F)-\psi(\tau)\ge  c(n,\tau)\,|t|^2\,.
  \end{equation}
  We now prove that \eqref{calibrazione t} holds even when $t>0$. In this case we use $w+K(\tau)\subset\{x_n>t\}$ and the inclusion $K(\tau)/2\subset F$ in \eqref{hp technical} to deduce that
  \[
  |F\Delta (w+K(\tau))|\ge|F\setminus(w+K(\tau))|\ge |F\setminus\{x_n>t\}|\ge \Big|\frac{K(\tau)}2\cap\{x_n\le t\}\Big|\,.
  \]
  By \eqref{hp technical}, \eqref{fimp} and the last inequality, if $\e(n,\tau)$ is small enough then we are able to infer that $t\le t_0(n,\tau)<1/2$ and then to exploit an elementary lower bound of the form
  \[
  |(K(\tau)/2)\cap\{x_n\le t\}|\ge c(n,\tau)\,t\,,
  \]
  for $t\in(0,t_0(n,\tau))\cc(0,1/2)$, to conclude that \eqref{calibrazione t} holds. This said, by combining the bound
  \begin{eqnarray*}
  |(z+K(\tau))\Delta F|&\le&|(w+K(\tau))\Delta F|+|(z+K(\tau))\Delta (w+K(\tau))|
  \\
  &\le&|(w+K(\tau))\Delta F|+C(n,\tau)\,|t|
  \end{eqnarray*}
  with \eqref{fimp} and \eqref{calibrazione t} we conclude the proof of \eqref{quantitative sessile problem}.
\end{proof}

\subsection{An improved convergence theorem}\label{section improved convergence} Thorough this section, we set $H=\{x_n>0\}$ and $K=K(\tau)$ for some fixed $\tau\in(-1,1)$, see \eqref{def K}. Our main result, Theorem \ref{thm improved convergence to K} below, consists in showing that if $F_h$ is sequence of almost-minimizing sets in $H$ which converges to $K$ in volume, then $M_h=\ov{H\cap\pa F_h}$ is a $C^{1,\a}$-hypersurface with boundary for $h$ large enough, and there exist $C^{1,\a}$-diffeomorphisms $f_h$ between $M_h$ and $M_0=\ov{H\cap\pa K}$ such that $f_h\to\Id$ in $C^1$ and enjoys certain precise structure properties. In order to formulate this result in rigorous terms we need to set some definitions.

\begin{definition}[Elliptic integrands]\label{def elliptic int}
  {\rm Given an open set $\Om$ one says that $\Phi$ is an {\it elliptic integrand on $\Om$} if $\Phi:\ov{\Om}\times\R^n\to[0,\infty]$ is lower semicontinuous with $\Phi(x,\cdot)$ convex and one-homogeneous on $\R^n$ for every $x\in \ov{\Om}$. If $F$ is of locally finite perimeter in $\Om$ and $W\subset\Om$ is a Borel set, then we define
  \begin{eqnarray}\label{PHI def}
  \PHI(F;W)=\int_{W\cap\pa^*F}\Phi(x,\nu_F(x))\,d\H^{n-1}(x)\in[0,\infty]\,.
  \end{eqnarray}
  Given $\l\ge1$ and  $\ell\ge0$ we let $\X(\Om,\l,\ell)$ denote the family of those elliptic integrands $\Phi$ in $\Om$ such that $\Phi(x,\cdot)\in  C^{2,1}(S^{n-1})$ for every $x\in\ov{\Om}$ and such that for every $x,x'\in\ov{\Om}$, $\nu,\nu'\in S^{n-1}$, one has
\begin{gather*}
  \frac1\l\le \Phi(x,\nu)\le\l\,,
  \\
  |\Phi(x,\nu)-\Phi(x',\nu)|+|\nabla\Phi(x,\nu)-\nabla\Phi(x',\nu)|\le\ell|x-x'|\,,
  \\
  |\nabla\Phi(x,\nu)|+\|\nabla^2\Phi(x,\nu)\|+\frac{\|\nabla^2\Phi(x,\nu)-\nabla^2\Phi(x,\nu')\|}{|\nu-\nu'|}\le\l\,,
  \\
  \big|\nabla^2\Phi(x,\nu)[\tau,\tau]\big|\ge\frac{|\tau|^2}\l\,,\qquad\forall \tau\in\nu^\perp\,,
\end{gather*}
where $\nabla\Phi$ and $\nabla^2\Phi$ are taken with respect to the $\nu$-variable, and with $\nu^\perp=\{y\in\R^n:y\cdot\nu=0\}$.
}
\end{definition}

The following minimality condition is tailored to the description of capillarity problems, in the sense that one considers subsets $F$ of an half-space which minimize an elliptic integrand with respect to local perturbations which are allowed to freely modify $\pa F\cap\pa H$. In other words, we impose a Dirichlet condition inside of $H$, and a Neumann/free-boundary condition on $\pa H$; see
\begin{figure}
  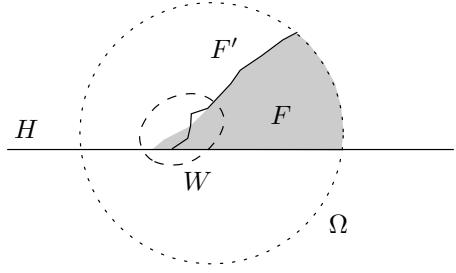\caption{\small{A competitor $F'$ of $F$ in \eqref{lambda minimiality Fk} can have a different trace along $\pa H$, but must agree with $F$ outside of an open set $W\cc\Om$ with small diameter. The boundary of $F'$ and $W$ are depicted, respectively, as a black line and as a dashed line.}}\label{fig almost}
\end{figure}
Figure \ref{fig almost}.

\begin{definition}[$(\Lambda,\rho_0)$-minimizers]\label{def lambda min}
  {\rm Let $\Om$ be an open set in $\R^n$, $H$ be an half-space in $\R^n$, and let $\Phi\in\X(\Om\cap H,\l,\ell)$. Given $\Lambda\ge0$ and $\rho_0>0$, a set $F\subset H$ of locally finite perimeter in $\Om$ is a $(\Lambda,\rho_0)$-minimizer of $\PHI$ in $(\Om,H)$ if
\begin{equation}
    \label{lambda minimiality Fk}
      \PHI(F;H\cap W)\le \PHI(F';H\cap W)+\Lambda\,|F\Delta F'|\,,
  \end{equation}
whenever $F'\subset H$ is such that $F\Delta F'\cc W$ with $W\cc \Om$ open and such that $\diam W<2\,\rho_0$.}
\end{definition}

We are now ready to state the following theorem, which is the main result of this section.

\begin{theorem}
  \label{thm improved convergence to K}
  Let $H=\{x_n>0\}$ and $K=K(\tau)$ for some fixed $\tau\in(-1,1)$. Given $\l\ge 1$, $\ell,\Lambda\ge0$, $\rho_0>0$, and $\a\in(0,1)$ there exists $C_\a$ depending on $n$, $\tau$, $\l$, $\ell$, $\Lambda$, $\rho_0$ and $\a$ with the following property.

  If $\Om$ is an open set such that $K\cc\Om$, $\{\Phi_h\}_{h\in\N}\subset\X(\Om,\l,\ell)$, and, for each $h\in\N$, $F_h$ is a $(\Lambda,\rho_0)$-minimizer of $\PHI_h$ in $(\Om,H)$ such that $|F_h\Delta K|\to 0$ as $h\to\infty$, then $M_h=\ov{H\cap\pa F_h}$ is a compact connected orientable $C^{1,\a}$-hypersurface with boundary for every $\a\in(0,1)$ and for $h$ large enough, and there exists a diffeomorphism $f_h:M_0\to M_h$, $M_0=\ov{H\cap\pa K}$ such that
  \[
  \|f_h\|_{C^{1,\a}(M_0)}\le C_\a\,,\qquad \lim_{h\to\infty}\|f_h-\Id\|_{C^1(M_0)}=0\,.
  \]
  Moreover, there exists $\e_h\to 0$ as $h\to\infty$ such that if we set
  \[
  u_h=(f_h-\Id)-[\nu_K\cdot(f_h-\Id)]\nu_K
  \]
  for the tangential displacement of $f_h$, then
  \[
  \spt\,u_h\subset M_0\cap\{x_n<\e_h\}\,,\qquad \lim_{h\to\infty}\|u_h\|_{C^1(M_0)}=0\,.
  \]
\end{theorem}

\begin{remark}
  {\rm Theorem \ref{thm improved convergence to K} works as well if we just assume that $M_0$ is a $C^{2,1}$-hypersurface with boundary such that $\bd(M_0)\subset\pa H$. We have decided to focus on the case relevant to the capillarity problem, where $M_0$ is a spherical cap, just for the sake of definiteness. Considering a sequence of elliptic energies $\Phi_h$ is necessary in view of the application of this result to sequences of $m_h^{1/n}$-blow-ups of global minimizers in \eqref{variational problem m} as $m_h\to 0^+$.}
\end{remark}

The rest of the section is devoted to the proof of Theorem \ref{thm improved convergence to K}. The following lemma allows us to exploit \cite[Theorem 3.5]{CiLeMaIC1} for constructing the diffeomorphisms $f_h$. Let us recall that if $M\subset\R^n$ is a $k$-dimensional manifold with boundary, then we denote by $\INT(M)$ and $\bd(M)$ its interior and boundary points respectively, by $\nu_M^{co}$ the outer unit normal to $\bd(M)$ in $M$, and we set
$[M]_\rho=M\setminus I_\rho(\bd(M))$ for every $\rho>0$.

\begin{lemma}\label{lemma hp thm3.5}
  Under the assumptions of Theorem \ref{thm improved convergence to K}, there exist positive constants $\mu_0\in(0,1)$ and $L$ (depending on $n$, $\tau$, $\l$, $\ell$, $\Lambda$, and $\rho_0$) and $L_\a$ (also depending on $\a\in(0,1)$) such that, for $h$ large enough and for every $\a\in(0,1)$, $M_h$ is a compact connected orientable $C^{1,\a}$-hypersurface with boundary with
  \[
  \bd(M_h)\ne\emptyset\,,\qquad\lim_{h\to\infty}\hd(M_h,M_0)=0\,.
  \]
  Furthermore, if we set $\nu_{M_h}=\nu_{F_h}$, then
  \begin{equation}\label{gre1}
  \left\{\begin{split}
    &|\nu_{M_h}(x)-\nu_{M_h}(y)|\le L\,|x-y|\,,
    \\
    &|\nu_{M_h}(x)\cdot(y-x)|\le L\,|x-y|^2\,,
  \end{split}\right .\qquad\forall x,y\in M_h\,.
  \end{equation}
  In addition, the following holds:

  \smallskip

  \noindent {\it (i)} if $n=2$, then $\bd(M_0)=\{p_0,q_0\}$ and $\bd(M_h)=\{p_h,q_h\}$, with
  \[
  \lim_{h\to\infty}|p_h-p_0|+|q_h-q_0|+|\nu_{M_h}^{co}(p_h)-\nu_{M_0}^{co}(p_0)|+|\nu_{M_h}^{co}(q_h)-\nu_{M_0}^{co}(q_0)|=0\,;
  \]
   if $n\ge 3$, then there exist $C^{1,\a}$-diffeomorphisms $f_{0,h}$ between $M_h$ and $M_0$ such that
  \begin{eqnarray*}
    \|f_{0,h}\|_{C^{1,\a}(\bd(M_0))}\le L_\a\,,
    \\
    \lim_{h\to\infty}    \|f_{0,h}-\Id\|_{C^1(\bd(M_0))}=0\,,
    \\
    \lim_{h\to\infty}    \|\nu_{M_h}(f_{0,h})-\nu_{M_0}\|_{C^1(\bd(M_0))}=0\,,
    \\
    \lim_{h\to\infty}    \|\nu_{M_h}^{co}(f_{0,h})-\nu_{M_0}^{co}\|_{C^1(\bd(M_0))}=0\,;
  \end{eqnarray*}
  where, by definition, $\nu_{M_0}=\nu_K$.

  \smallskip

  \noindent (ii) for every $\rho<\mu_0^2$ there exist $h(\rho)\in\N$ and $\{\psi_h\}_{h\ge h(\rho)}\subset C^{1,\a}([M_0]_\rho)$ such that
  \begin{equation}\label{basta iii j}
    \begin{split}
      &\hspace{0.9cm}[M_h]_{3\rho}\subset(\Id+\psi_h\,\nu_{M_0})([M_0]_\rho)\subset M_h\,,\qquad\forall h\ge h(\rho)\,,
      \\
      &\sup_{h\ge h(\rho)}\|\psi_h\|_{C^{1,\a}([M_0]_\rho)}\le L_\a\,,\qquad \lim_{h\to\infty}\|\psi_h\|_{C^1([M_0]_\rho)}=0\,.
    \end{split}
  \end{equation}
\end{lemma}

\begin{proof}
  [Proof of Theorem \ref{thm improved convergence to K}] By combining Lemma \ref{lemma hp thm3.5} with \cite[Theorem 3.5]{CiLeMaIC1} we find that for every $\mu\in(0,\mu_0)$ there exist $h(\mu)\in\N$ and, for each $h\ge h(\mu)$, a $C^{1,\a}$-diffeomorphisms $f_h$ between $M_0$ and $M_h$ such that
  \begin{equation}
    \label{jjj}
      \begin{split}
    &\hspace{1.5cm}\mbox{$f_h=f_{0,h}$ on $\bd(M_0)$}\,,\qquad
    \mbox{$f_h=\Id+\psi_h\,\nu_{M_0}$ on $[M_0]_\mu$}\,,
    \\
    &\hspace{1cm}\sup_{h\ge h(\mu)}\|f_h\|_{C^{1,\a}(M_0)}\le C_\a\,,\qquad \lim_{h\to\infty}\|f_h-\Id\|_{C^1(M_0)}=0\,,
    \\
    &\|u_h\|_{C^1(M_0)}\le\frac{C_\a}\mu\,\left\{
    \begin{split}
    \|(f_{0,h}-\Id)\cdot\nu_{M_0}^{co}\|_{C^0(\bd(M_0))}\,,\qquad\mbox{if $n=2$}\,,
    \\
    \|f_{0,h}-\Id\|_{C^1(\bd(M_0))}\,,\qquad\mbox{if $n\ge 3$}\,,
    \end{split}\right .
  \end{split}
  \end{equation}
  where, in the case $n=2$, $f_{0,h}$ is defined by the relations $f_{0,h}(p_0)=p_h$, $f_{0,h}(q_0)=q_h$. Of course \eqref{jjj} implies the conclusions of Theorem \ref{thm improved convergence to K}.
\end{proof}

We now focus on the proof of Lemma \ref{lemma hp thm3.5}. The proof is based on the regularity theory for $(\Lambda,\rho_0)$-minimizers of elliptic integrands as discussed in \cite[Part III]{maggiBOOK} and in \cite{dephilippismaggiARMA}. Given the regularity theorems, the argument is rather standard and so we limit ourselves to describe the main points.

\begin{proof}
  [Proof of Lemma \ref{lemma hp thm3.5}] One deduces that $\hd(M_h,M_0)\to 0$ by $|F_h\Delta K|\to 0$ and by the uniform density estimates satisfied by the sets $F_h$ (recall that the functionals $\PHI_h$ are uniformly elliptic), see for example \cite[Theorem 2.9]{dephilippismaggiARMA}. Let us now set
   \[
  {\rm Reg}(M_h)=\bigg\{x\in M_h:
  \begin{array}{l}
  \textrm{there exists \(r_x>0\) such that \(M_h\cap B_{x,r_x}\)}
  \\
  \textrm{is a \(C^1\)-manifold with boundary  s.t. $\bd(M_h)\subset\pa H$}
  \end{array}\bigg\}\,,
 \]
 for the  regular part of $M_h$ and $\S(M_h)=M_h\setminus{\rm Reg}(M_h)$ for its singular set. For $x\in \Omega$ and $r<\dist(x,\pa \Omega)$, define the {\it spherical excess of $F_h$ at the point $x$, at scale $r$, relative to $H$} as
 \[
 {\bf exc}^H(F_h,x,r)=\inf\Big\{\frac1{r^{n-1}}\int_{B_{x,r}\cap H\cap\pa^*F_h}\frac{|\nu_{F_h}-\nu|^2}2\,d\H^{n-1}:\nu\in S^{n-1}\Big\}\,.
 \]
 By (for example) \cite[Theorem 3.1]{dephilippismaggiARMA} in the case $x\in M_h\cap\pa H$, and \cite{DuzaarSteffen} in the case $x\in M_h\cap H$, there exists $\e_0=\e_0(n,\lambda)>0$ such that
 \begin{equation}\label{singular set char}
 \S(M_h)=\Big\{x\in M_h:\liminf_{r\to 0^+}{\bf exc}^H(F_h,x,r)\ge\e_0\Big\}\,\,.
 \end{equation}
 This characterization of the singular set allows to deduce easily that for every $\tau>0$ one has $\S(M_h)\subset I_\tau(\S(M_0))$ provided $h$ is large enough. In our case $M_0$ is just a spherical cap, and so we have $\S(M_0)=\emptyset$. In particular, $M_h$ is a $C^{1,\a}$-hypersurface with boundary for every $\a\in(0,1)$.

 In fact one has more precise information. We first describe the argument qualitatively. Consider for example a point $x\in M_0\cap H$, and fix $r>0$ such that $\dist(x,\pa H)>r$ and ${\bf exc}^H(K,x,r)<\e_0/2$. The Hausdorff convergence of $M_h$ to $M_0$, and the fact that almost-minimizers converging in volume also converge in perimeter, implies the existence of $x_h\in M_h\cap H$ such that $x_h\to x$ and ${\bf exc}^H(F_h,x_h,r)<\e_0$. One can thus apply the $\e$-regularity criterion to $K$ and to $F_h$ to find that they are both epigraphs of $C^{1,\a}$-functions $v_h$ and $v$ defined on a same $(n-1)$-dimensional disk or radius $c_0\,r$, with $v_h\to v$ in $C^1$ and $c_0=c_0(n,\a)$. By patching this local graphicality property on a uniform scale one come to prove conclusion (ii) in the lemma. This kind of argument is described in great detail in \cite[Theorem 4.1, Lemma 4.3, Lemma 4.4, Theorem 4.12]{CiLeMaIC1} and so we do not further discuss this point.

 We now exploit the same argument at the boundary. Given $x\in\pa H$, $r>0$ and $\nu\in S^{n-1}$ with $\nu\cdot\nu_H=0$, set
 \[
 \D^\nu_{x,r}=\big\{y\in H:|(y-x)\cdot\nu|=0\,,|y-x|<r\big\}\,,
 \qquad
 \C^\nu_{x,r}=\big\{y+t\nu:y\in\D^\nu_{x,r},|t|<r\big\}\,.
 \]
 Let us fix $x\in M_0\cap\pa H$, and consider $r_x>0$ such that ${\bf exc}^H(K,x,r_x)<\e_0/2$. By exploiting a {\it boundary} $\e$-regularity criterion \cite[Theorem 3.1]{dephilippismaggiARMA}, we come to prove that for every $x\in M_0\cap\pa H$ there exist $\nu_x\in S^{n-1}$ with $\nu_x\cdot\nu_H=0$ and functions $v_h,v\in C^{1,\a}(\D_{x,c_0\,r_x}^{\nu_x})$ such that $v_h\to v$ in $C^1(\D_{x,c_0\,r_x}^{\nu_x})$ with
 \[
 \C_{x,c_0\,r_x}^{\nu_x}\cap F_h=\Big\{y\in \C_{x,c_0\,r_x}^{\nu_x}:y\cdot\nu_x>v_h\big(y-(y\cdot\nu_x)\nu_x\big)\Big\}\,,
 \]
 and with an analogous relation between $K$ and $v$. In particular,
 \begin{equation}\label{patch}
   \begin{split}
   \C_{x,c_0\,r_x}^{\nu_x}\cap M_h\cap\pa H&=\Big\{z+v_h(z)\,\nu_x:z\in\ov{\D_{x,c_0\,r_x}^{\nu_x}}\cap\pa H\Big\}\,,
   \\
   \C_{x,c_0\,r_x}^{\nu_x}\cap M_0\cap\pa H&=\Big\{z+v(z)\,\nu_x:z\in\ov{\D_{x,c_0\,r_x}^{\nu_x}}\cap\pa H\Big\}\,.
   \end{split}
 \end{equation}
 Let us now cover the $(n-2)$-dimensional sphere $M_0\cap\pa H$ by finitely many cylinders satisfying \eqref{patch}. By exploiting the fact that $v_h\to v$ in $C^1$ (see \cite[Lemma 4.3]{CiLeMaIC1} for the details of such a construction) we can show that the normal projection over $M_0\cap\pa H$ defines a $C^{1,\a}$-diffeomorphism between $M_h\cap\pa H$ and $M_0\cap\pa H$. Denoting by $f_{0,h}$ the inverse of this map, we complete the proof of conclusion (i).

 Summarizing we have proved the validity of conclusions (i) and (ii). The fact that $M_h$ is compact and connected is also easily inferred by covering a neighborhood of $M_0$ by finitely many cylinders of graphicality for both $M_h$ and $M_0$ and by recalling that $\hd(M_h,M_0)\to 0$. Let us finally consider the vector fields $\nu_{M_h}=\nu_{F_h}$ and $\nu_K$, and notice that
 \begin{equation}\label{gre2}
  \left\{\begin{split}
    &|\nu_K(x)-\nu_K(y)|\le L\,|x-y|\,,
    \\
    &|\nu_K(x)\cdot(y-x)|\le L\,|x-y|^2\,,
  \end{split}\right .\qquad\forall x,y\in K\,,
  \end{equation}
  for a suitably large constant $L$. By arguing as in \cite[Theorem 26.6]{maggiBOOK} we see that if $x_h\in M_h$ and $x_h\to x$, then $x\in M_0$ and $\nu_{M_h}(x_h)\to\nu_K(x)$; by exploiting this fact and \eqref{gre2} we easily deduce \eqref{gre1}.
\end{proof}

\section{Convergence in volume to the ideal droplet}\label{section convergence to K} Throughout this section $A$ denotes an open bounded connected set with boundary of class $C^{1,1}$, while $\s:\pa A\to(-1,1)$ and $g:A\to[0,\infty)$ are Lipschitz functions. The minimum value on $\pa A$ of $\s$ is denoted by $\s_0$,
\[
 \s_0=\min_{\pa A}\s\,.
\]
and we denote by $K=K(\s_0)$ the reference unit volume droplet associated to $\F_{H,\s_0}$ as in \eqref{def K}, where $H=\{x\in\R^n:x_n>0\}$. The goal of this section is showing that minimizers $E_m$ in the variational problem \eqref{variational problem m}, which we recall was defined by
\[
\g(m)=\inf\Big\{\F_{A,\s}(E)+\int_E\,g(x)\,dx:E\subset A\,,|E|=m\Big\}\,,
\]
are such that $\psi_m(E_m)/|\psi_m(E_m)|^{1/n}\to K$ in volume as $m\to 0^+$. Here the maps $\psi_m$ are defined on neighborhoods (of uniformly positive diameter) of $E_m$ and converge in $C^{1,1}$ to the identity map.

\begin{lemma}\label{thm convergenza a K}
  If $A$ is an open bounded connected set of class $C^{1,1}$, $\s\in\Lip(\pa A;(-1,1))$ and $g\in \Lip(A)$ with $g\ge 0$, then there exist positive constants $C_0$ and $m_0$ (depending on $A$, $\s$ and $g$) such that if $E_m$ is a minimizer in \eqref{variational problem m} with $m<m_0$, then there exists $y_m\in\pa A$ such that
  \begin{equation}
    \label{boundedness}
      E_m\subset B_{y_m,C_0\,m^{1/2n}}\qquad 0\le \s(y_m)-\s_0\le C_0\,m^{1/2n}\,.
  \end{equation}
  Moreover, for every $\e>0$ there exists $m_\e\le m_0$ (depending on $A$, $\s$, $g$ and $\e$) such that
  \begin{equation}
    \label{proximity epsilon}
      \inf_{z\in\pa H}|F_m\Delta (z+K)|<\e\,,\qquad\mbox{where}\qquad F_m=\frac{\phi^{-1}_{y_m}(E_m)}{|\phi^{-1}_{y_m}(E_m)|^{1/n}}\,,
  \end{equation}
  where $\phi_{y_m}$ is defined as in Notation \ref{notation boundary of A} below.
\end{lemma}

\begin{notation}\label{notation boundary of A}
{\rm  (See
\begin{figure}
  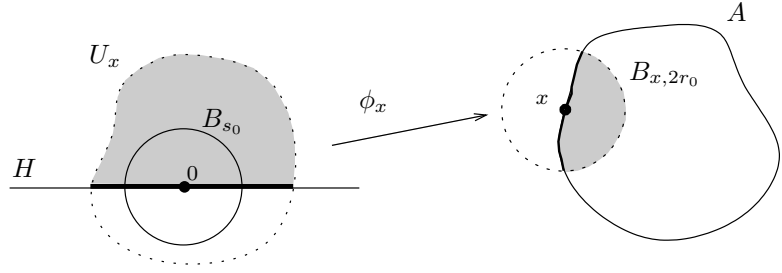\caption{{\small A summary of Notation \ref{notation boundary of A}. The map $\phi_x$ is such that $\phi_x(0)=x$. The constants $r_0$ and $s_0$ are independent of $x\in\pa A$.}}\label{fig notation}
\end{figure}
Figure \ref{fig notation}.) By compactness of $\pa A$, we can find $r_0,s_0>0$ (depending on $A$) so that for every $x\in\pa A$ there exist an open neighborhood $U_x$ of $0\in\R^n$ with $\ov{B}_{s_0}\subset U_x$ and a $C^{1,1}$-diffeomorphism $\phi_x:U_x\to B_{x,2r_0}$ such that $\phi_x(0)=x$, $\nabla\phi_x(0)$ is an orientation preserving isometry, and
\begin{equation}
    \label{def phix}
      \phi_x(U_x\cap H)=B_{x,2r_0}\cap A\,,\qquad \phi_x(U_x\cap\pa H)=B_{x,2r_0}\cap\pa A\,.
\end{equation}
Thanks to the fact that $L=\nabla\phi_x(0)$ is an orientation preserving isometry one has $L^{-1}=L^*$, $\det\,L=1$, $\cof(L)=[\det(L)L^{-1}]^*=L$, and thus the Jacobian of $\nabla\phi_x(0)$ and its tangential Jacobian on ant hyperplane $\nu^\perp$ corresponding to $\nu\in S^{n-1}$ are given by
\begin{equation}
  \label{marina1}
  J\phi_x(0)=|\det\nabla\phi_x(0)|=1\,,\qquad J^{\nu^\perp}\phi_x(0)=|\cof(\nabla\phi_x(0))\nu|=1\,.
\end{equation}
In particular, there exists a constant $C_1$ depending only on $A$ such that
\begin{equation}\label{jacobian phi estimates}
  \|\phi_x\|_{C^{1,1}(U_x)}+\|\phi_x^{-1}\|_{C^{1,1}(B_{x,2\,r_0})}\le C_1\,,\qquad
  \left\{\begin{split}
     &\|J\phi_x-1\|_{C^0(B_s)}\le C_1\,s\,,
     \\
     &\|J^\Sigma\phi_x-1\|_{C^0(\S\cap B_s)}\le C_1\,s\,,
  \end{split}
  \right .
\end{equation}
for every $s<s_0$ and for every $(n-1)$-rectifiable set $\Sigma$ in $\R^n$. We can also assume $C_1\,s_0$ as small as needed depending on $A$. For example, we can certainly entail $J\phi_x\ge1/2$ on $B_{s_0}$.}
\end{notation}

\begin{remark}
  {\rm We recall that $O(m^{1/2n})$ in \eqref{boundedness} is not optimal, and that it will be improved to $O(m^{1/n})$ in the next section. By being able of immediately obtain the latter information we could simplify some technicalities in Lemma \ref{lemma almost minimizers} below. However, this limitation seems an unavoidable consequence of the grid argument used in step five below.}
\end{remark}

\begin{proof}[Proof of Lemma \ref{thm convergenza a K}]
  In the course of the argument, $C$ denotes a generic constant whose value depends on $A$, $g$ and $\s$. The values $r_0$ and $s_0$ introduced in Notation \ref{notation boundary of A} as $A$-dependent constants will be further decreased depending on $A$, $g$ and $\s$.

  \medskip

  \noindent {\it Step one}: We show that
  \begin{equation}
    \label{upper bound gamma m}
      \g(m)\le \psi(\s_0)\,m^{(n-1)/n}\Big(1+C\,m^{1/n}\Big)\,,\qquad\forall m<m_0\,.
  \end{equation}
  To this end, given $m<m_0$ it suffices to construct $E\subset A$ such that
  \[
  |E|=m\,,\qquad \F_{A,\s}(E)\leq \psi(\s_0)\, m^{(n-1)/n} (1+ C m^{1/n})\,.
  \]
  Let us fix $x\in\pa A$ such that $\s(x)=\s_0=\min_{\pa A}\s$, and correspondingly set $U=U_{x}$ and $\phi=\phi_{x}$. We can find $t_0>0$ (depending on $A$ and $\s_0$) such that $K_t=t\,K\subset B_{s_0}$ for every $t<t_0$, and thus, by \eqref{def phix},
  \[
  E(t)=\phi(K_t)\subset B_{x,2\,r_0}\cap A\,,\qquad\forall t\in(0,t_0)\,.
  \]
  By the area formula $|E(t)|=\int_{K_t}\,J\phi$, and since $J\phi\ge 1/2$ on $B_{s_0}$ by \eqref{jacobian phi estimates}, we find that $t\in(0,t_0)\mapsto |E(t)|$ is strictly increasing. In particular, we can find $m_0$ such that
  \[
  (0,m_0)=\{|E(t)|:t\in(0,t_0)\}\,,
  \]
  and for every $m<m_0$ there exists a unique $t(m)< t_0$ such that $m=|E(t(m))|$. We notice that $t(m)\le C\,m^{1/n}$: indeed, since $K_t\subset B_{C\,t}$ and $|K_t|=t^n$ for every $t>0$, by \eqref{jacobian phi estimates} we find
  \[
  m=|E(t(m))|=\int_{K_{t(m)}}J\phi\ge (1-C\,t(m))\,|K_{t(m)}|=(1-C\,t(m))\,t(m)^n\ge\frac{t(m)^n}2\,,
  \]
  where the last inequality follows up to further decreasing the value of $t_0$ (depending on $A$ and $\s_0$). By the area formula,
  \begin{eqnarray*}
    \F_{A,\s}(E(t))+\int_{E(t)}g=\int_{H\cap\pa K_t}J^{\pa K_t}\,\phi\,d\H^{n-1}+
    \int_{\pa H\cap\pa K_t}\s(\phi)\,J^{\pa K_t}\,\phi\,d\H^{n-1}+\int_{K_t}g(\phi)\,J\phi\,,
  \end{eqnarray*}
  so that \eqref{jacobian phi estimates}, $\F_{H,\s_0}(K_t)=t^{n-1}\,\psi(\s_0)$ for every $t>0$, and
  \begin{equation}\label{sigma Lipschitz phi}
  \|\s\circ\phi_x-\s(x)\|_{C^0(\pa H\cap B_s)}\le C\,s\,,\qquad\forall s<s_0\,.
  \end{equation}
  give us
  \begin{eqnarray*}
    \F_{A,\s}(E(t))+\int_{E(t)}g&\le&(1+C\,t)\,P(K_t,H)+(\s_0+C\,t)\,\H^{n-1}(\pa H\cap\pa K_t)+C\,|K_t|
    \\
    &=&\F_{H,\s_0}(K_t)+C\,t\,P(K_t)+C\,t^n\le \psi(\s_0)\,t^{n-1}+C\,t^n\,,
  \end{eqnarray*}
 which combined with $t(m)\le C\,m^{1/n}$ leads to \eqref{upper bound gamma m}.

 \medskip

 \noindent {\it Step two}: We show that, if $E\subset A$ with $\diam(E)<r_0$, then
  \begin{equation}
    \label{lower bound FAsE small diameter}
      \F_{A,\s}(E)\ge\psi(\s_0)\,|E|^{(n-1)/n}\,(1-C\,\diam(E))\,.
  \end{equation}
  Indeed let $r=\diam(E)$. If $E\subset B_{x,r}$ for some $x\in A$ with $\dist(x,\pa A)>r$, then $\F_{\s,A}(E)=P(E)$, and thus
  \begin{equation}
    \label{step two 1}
      \F_{\s,A}(E)=P(E)\ge n\om_n^{1/n}\,|E|^{(n-1)/n}\ge\psi(\s_0)\,|E|^{(n-1)/n}\,,
  \end{equation}
  where we have applied the isoperimetric inequality and the fact that $n\om_n^{1/n}=\psi(1)\ge \psi(\s_0)$ thanks to Proposition \ref{lemma psi tau}-(i). Since \eqref{step two 1} implies \eqref{lower bound FAsE small diameter}, and since $E\subset A$ with $\diam(E)<r$, we are left to consider the case when
  \[
  E\subset B_{x,2r}\,,\qquad\mbox{for some $x\in\pa A$}\,.
  \]
  Set $U=U_x$ and $\phi=\phi_x$. Since $r<r_0$ it makes sense to define $F=\phi^{-1}(E)$. Since $F\subset B_{C\,r}$, by the area formula and by \eqref{jacobian phi estimates} one has
 \begin{eqnarray*}
  \F_{A,\s}(E)&=&\int_{H\cap\pa^* F}J^{\pa^* F}\,\phi\,d\H^{n-1}+\int_{\pa H\cap\pa^* F}\s(\phi)\,J^{\pa^* F}\,\phi\,d\H^{n-1}
  \\
  &\ge& (1-C\,r)\,P(F;H)+(\s(x)-C\,r)\,\H^{n-1}(\pa H\cap\pa^*F)\ge \F_{H,\s_0}(F)-C\,r\,P(F)\,.
 \end{eqnarray*}
 By \cite[Proposition 19.22]{maggiBOOK} we have $\F_{H,\s_0}(F)\ge (1+\s_0)P(F)/2$, thus
 \[
 \F_{A,\s}(E)\ge(1-C\,r)\,\F_{H,\s_0}(F)\ge(1-C\,r)\,\psi(\s_0)\,|F|^{(n-1)/n}\,.
 \]
 Finally, since $|E|=\int_FJ\phi$ implies $|F|^{(n-1)/n}\ge|E|^{(n-1)/n}(1-C\,r)$, one finds \eqref{lower bound FAsE small diameter}.

 \medskip

 \noindent {\it Step three}: We prove that if $m<m_0$, then
 \begin{equation}
   \label{upper bound PEm}
   P(E_m)\le C\,m^{(n-1)/n}\,.
 \end{equation}
 We are going to deduce this from \eqref{upper bound gamma m} and the general inequality
 \begin{equation}
   \label{general inequality}
   P(E)\le C\,\F_{A,\s}(E)\,,\qquad\forall E\subset A\,,|E|=m<m_0\,.
 \end{equation}
 Let us prove \eqref{general inequality}. This is obvious if $\s_0>0$, as in this case,
 \begin{eqnarray*}
  P(E)=P(E;A)+P(E;\pa A)\leq \s_0^{-1}\big( P(E;A)+\s_0\,P(E;\pa A)\big)\le\s_0^{-1}\,\F_{A,\s}(E)\,.
 \end{eqnarray*}
 Let us now assume that $\s_0\le0$, and consider the relative isoperimetric problems
\begin{eqnarray*}
  \l_A(m)&=&\inf\Big\{\frac{P(E;A)}{\H^{n-1}(\pa^* E\cap \pa A)}:|E|=m,E\subset A\Big\}\,,
  \\
  \qquad
  \mu_A&=&\inf\Big\{\frac{P(E;A)}{|E|^{n/(n-1)}}:|E|\le\frac{|A|}2,\,E\subset A\Big\}\,,
\end{eqnarray*}
  so that $\mu_A>0$ (as $A$ is a connected bounded open set with Lipschitz boundary). Let us first show that for all $\tau>0$ there is $m_1$ (depending on $A$ and $\tau$) such that
  \begin{equation}
    \label{lambda A m}
    \l_A(m)\ge\frac1{1+\tau}\,,\qquad\forall m<m_1\,.
  \end{equation}
 Indeed, since $\pa A$ is of class $C^{1,1}$, one can construct $T\in{\rm Lip}(\R^n;\R^n)$ with $T=-\nu_A$ on $\pa A$ and $|T|\le 1$ on $\R^n$. By the divergence theorem, if $E$ is a competitor in $\l_A(m)$, then
 \[
 \H^{n-1}(\pa^* E\cap\pa A)\le P(E;A)+\Big|\int_E\Div T\Big|\le P(E;A)+\Lip(T)\,m\le \Big(1+\,\frac{\Lip(T)\,m^{1/n}}{\mu_A}\Big)\,P(E;A)\,,
 \]
 so that \eqref{lambda A m} follows immediately. Let us now consider $\tau>0$ and $a>0$ such that
 \begin{equation}
   \label{step three 1}
    (1+\tau)\,a\,\s_0=-(a+\s_0-1)\,.
 \end{equation}
 (The reason for this choice will become apparent in a moment.) Notice that \eqref{step three 1} is equivalent to requiring
 \[
 a=\frac{1-\s_0}{1+(1+\tau)\s_0}>0\,,
 \]
 which is certainly possible (by $\s_0>-1$) as soon as $\tau$ is small enough depending on $\s_0$. Let $m_0\le m_1$ for $m_1=m_1(A,\tau)$. By \eqref{lambda A m}, if $E\subset A$ and $|E|=m<m_0$, then
\begin{eqnarray*}
P(E) & = & P(E;A)+(1-\s_0)\H^{n-1}(\pa^* E\cap \pa A)+\s_0\H^{n-1}(\pa^* E\cap \pa A)
\\
& = & \F_{A,\s_0}(E)+(a-(a +\s_0-1))\H^{n-1}(\pa^* E\cap \pa A)
\\
& \le & \F_{A,\s_0}(E)+\frac{a}{\l_A(m)}\,P(E;A) -(a +\s_0-1)\H^{n-1}(\pa^* E\cap \pa A)
\\
& \le & \F_{A,\s_0}(E)+(1+\tau)a\,P(E;A) -(a +\s_0-1)\H^{n-1}(\pa^* E\cap \pa A)
\\
&=& (1+(1+\tau)a)\F_{A,\s_0}(E) \le (1+(1+\tau)\,a)\,\F_{A,\s}(E)\,,
\end{eqnarray*}
where in the identity on the last line we have used \eqref{step three 1}. This completes the proof of \eqref{general inequality}.

\medskip

\noindent {\it Step four}: With the goal in mind of bounding the diameter of $E_m$ in terms of $m$, we now  estimate the normalized volume error one makes in boxing $E_m$ into a (properly centered) cube $Q_r$ of side length $r$. More precisely, we show the existence of positive constants $C_2$ and $r_1\le r_0$ (depending on $A$, $\s$ and $g$) such that if $m<m_0$ and $r<r_1$, then
\begin{equation}
    \label{r estimate}
\frac{|E_m\setminus Q_r|}m\le C_2\,\Big(\frac{m^{1/n}}r+r\Big)^{n/(n-1)}
\end{equation}
for a cube $Q_r$ of side length $r$. Indeed, by applying \cite[Lemma 5.1]{FigalliMaggiLOGCONV} to $E_m^{(1)}$ (the set of points of density one of $E_m$ in $\R^n$), we find that for every $r>0$ there is a partition of (Lebesgue almost all of) $\R^n$ into a family $\mathcal{Q}$ of open parallel cubes with side length $r$ such that
\begin{equation}\label{ineq2}
\frac{|E_m|}{r}=\frac{|E_m^{(1)}|}{r}\ge\frac1{2n}\,\sum_{Q\in\mathcal{Q}}\H^{n-1}(E_m^{(1)}\cap\pa Q)\,,
\end{equation}
and it is actually clear from the proof of that lemma that these cubes can be chosen so that
\begin{equation}
  \label{cubes}
  \H^{n-1}(\pa Q\cap\pa A)=\H^{n-1}(\pa Q\cap\pa^*E_m)=0 \qquad\mbox{for every $Q\in\mathcal{Q}$}\,.
\end{equation}
If $r_1$ is small enough in terms of $A$, $\s$ and $g$ and $C$ is the constant appearing on the left-hand side of \eqref{lower bound FAsE small diameter} (see step two), then we can entail $(1-C\,r)>0$ with $(1-C\,r)^{-1}\le 1+2\,C\,r$ for $r<r_1$. Up to also requiring that $r_1<r_0/\sqrt{n}$, by applying \eqref{lower bound FAsE small diameter} to $E=E_m\cap Q$ we find that
\begin{equation}
  \label{adding up}
  \psi(\s_0)\,|E_m\cap Q|^{(n-1)/n}\le (1+2\,C\,r)\,\F_{A,\s}(E_m\cap Q)\,.
\end{equation}
By \cite[Equation (16.4)]{maggiBOOK}, $E_m^{(1)}\cap A=E_m^{(1)}$, $\H^{n-1}(\pa^*Q\Delta\pa Q)=0$, and \eqref{cubes}, we have
\[
P(E_m\cap Q;A)=P(E_m;Q\cap A)+\H^{n-1}(\pa Q\cap E_m^{(1)})\,,\qquad\forall Q\in\mathcal{Q}\,,
\]
and thus, taking also into account the $\H^{n-1}$-equivalence of the sets $\pa A\cap\pa^*(E_m\cap Q)$ and $Q\cap\pa A\cap\pa^*E_m$ (which follows by \eqref{cubes}),
\[
\F_{A,\s}(E_m\cap Q)= P(E_m;Q\cap A)+\int_{Q\cap\pa A\cap\pa^*E_m}\s+\H^{n-1}(\pa Q\cap E_m^{(1)})\,.
\]
By this last identity, by adding up over $Q$ in \eqref{adding up}, and by \eqref{ineq2} one finds
\begin{eqnarray*}
\psi(\s_0)\,\sum_{Q\in\mathcal{Q}}|E_m\cap Q|^{(n-1)/n} & \le & (1+Cr)\sum_{Q\in\mathcal{Q}}\F_{A,\s}(E_m\cap Q)
\\
& = & (1+Cr)\Big(\F_{A,\s}(E_m) + \sum_{Q\in\mathcal{Q}} \H^{n-1}(E_m^{(1)}\cap \pa Q) \Big)
\\ 
& \le & (1+Cr)\Big(\g(m)-\int_{E_m} g+ 2n\,\frac{m}{r} \Big)
\\
&\le& (1+C\,r)\Big(\g(m)+C\,\frac{m}r\Big)\,.
\end{eqnarray*}
By step one, see \eqref{upper bound gamma m},
\begin{eqnarray*}
\psi(\s_0)\,\sum_{Q\in\mathcal{Q}}|E_m\cap Q|^{(n-1)/n}
\le C\,\frac{m}{r} + (1+C\,r)\,\psi(\s_0)\,m^{(n-1)/n}\,(1+C\,m^{1/n})\,,
\end{eqnarray*}
so that, first dividing both sides by $m^{(n-1)/n}$, and then subtracting $\psi(\s_0)$, we have
\begin{equation}\label{gives}
  \begin{split}
    \psi(\s_0)\Big(\sum_{Q\in\mathcal{Q}}\Big(\frac{|E_m\cap Q|}{m}\Big)^{(n-1)/n} -1\Big)&\le
    C\,\frac{m^{1/n}}{r} +  C\,(r+m^{1/n})
    \\
    &\le C\,\Big(\frac{m^{1/n}}r+r\Big)\,.
  \end{split}
\end{equation}
Now let $\{Q_i\}_{i=1}^N=\{Q\in\mathcal{Q}:|E_m\cap Q|>0\}$. (Notice that $N<\infty$ as only finitely many cubes from $\mathcal{Q}$ can intersect $A$, thus $E_m$.) We can order these cubes so that
\[
\Big|E_m\cap\bigcup_{i=1}^{M-1}Q_i\Big|<\frac{|E_m|}2\,,\qquad \Big|E_m\cap\bigcup_{i=M+1}^N Q_i\Big|<\frac{|E_m|}2\,.
\]
By concavity,
\[
\sum_{i=1}^N\Big(\frac{|E_m\cap Q_i|}{m}\Big)^{(n-1)/n}\ge\Big(\sum_{i=1}^{M-1}\frac{|E_m\cap Q_i|}{m}\Big)^{(n-1)/n}+
\Big(\sum_{i=M}^N\frac{|E_m\cap Q_i|}{m}\Big)^{(n-1)/n}\,,
\]
and since $t^{(n-1)/n}+(1-t)^{(n-1)/n}-1\ge c(n)\,t^{(n-1)/n}$ for every $t\in[0,1/2]$, \eqref{gives} gives
\[
\Big(\sum_{i=1}^{M-1}\frac{|E_m\cap Q_i|}{m}\Big)^{(n-1)/n}\le  C\,\Big(\frac{m^{1/n}}r+r\Big)\,,
\]
and, by an analogous argument,
\[
\Big(\sum_{i=M+1}^N\frac{|E_m\cap Q_i|}{m}\Big)^{(n-1)/n}\le  C\,\Big(\frac{m^{1/n}}r+r\Big)\,.
\]
In conclusion,
\[
\frac{|E_m\setminus Q_M|}m=\sum_{i\ne M}\frac{|E_m\cap Q_i|}{m}\le  C\,\Big(\frac{m^{1/n}}r+r\Big)^{n/(n-1)}\,,
\]
and \eqref{r estimate} is proved.

\medskip

\noindent {\it Step five}: From now on we shall always assume that
\begin{equation}
  \label{m0 r1}
  m_0\le r_1^{2n}\,.
\end{equation}
with $r_1$ as in the previous step. We claim that for every $m<m_0$ there exists $x_m\in A$ such that
\begin{equation}
  \label{r estimate 2}
  \big|E_m\setminus B_{x_m,C(n)m^{1/2n}}\big|\le C_3\,m^{1+[1/2(n-1)]}\,,
\end{equation}
where $C_3=C_3(A,\s,g)$. To prove this we apply step four with $r=m^{1/2n}$ (as we can as $r<r_1$ by \eqref{m0 r1}) to find a cube $Q_r$ of side length $r$ such that $Q_r\cap A\ne\emptyset$ and
\[
|E_m\setminus Q_r|\le C_2\,m\,\Big(r+\frac{m^{1/n}}r\Big)^{n/(n-1)}\le C_3\,m^{1+[1/2(n-1)]}\,.
\]
Since $Q_r\cap A\ne\emptyset$, there exists $x_m\in A$ such that $Q_r\subset B_{x_m,C(n)\,r}$, and thus \eqref{r estimate 2} is proved.

\medskip

\noindent {\it Step six}: We prove if $m<m_0$ and if $x_m$ is defined as in step five, then
  \begin{equation}
  \label{boundedness 2}
  E_m\subset B_{x_m,C\,m^{1/2n}}\,,
  \end{equation}
  for a constant $C=C(A,\s,g)$. To prove our assertion it suffices to show that given a sequence $m_h\to 0^+$, then, possibly up to extracting subsequences, one has
  \begin{equation}
  \label{boundedness 2h}
  E_h\subset B_{x_h,C\,m_h^{1/2n}}\,,\qquad\mbox{where}\quad E_h=E_{m_h}\,,\quad x_h=x_{m_h}\,,
  \end{equation}
  for some constant $C$. Up to subsequences we may assume that $x_h\to x_0$ for some $x_0\in\ov{A}$. We now distinguish two cases.

  \smallskip

  \noindent {\it Case one, $x_0\in A$}: We set, for $C(n)$ as in the left-hand side of \eqref{r estimate 2},
  \[
  I=\Big(C(n)\,m_h^{1/2n},\frac{\dist(x_0,\pa A)}4\Big)\,,
  \]
  and, for every $r\in I$, we let
  \begin{equation}
    \label{Ehr 00}
      E_h^r=E_h\cap B_{x_h,r}\,,\qquad u_h(r)=\frac{|E_h\setminus B_{x_h,r}|}{m_h}\,,\qquad \l_h(r)=\Big(\frac{m_h}{|E_h^r|}\Big)^{1/n}\,.
  \end{equation}
  If $\hat{E}_h^r$ denotes the dilation of $E_h^r$ by a factor $\l_h(r)$ with respect to the point $x_0$, then
  \begin{equation}
    \label{Ehr 1}
      |\hat{E}_h^r|=m_h\,,\qquad \hat{E}_h^r\subset B_{x_0,\l_h(r)\,\dist(x_0,\pa A)/2}\qquad\forall r\in I\,,
  \end{equation}
  provided $h$ is large enough. Indeed the inclusion follows by noticing that, for $h$ large,
  \[
  E_h^r\subset B_{x_h,r}\subset B_{x_0,r+|x_h-x_0|}\subset B_{x_0,\dist(x_0,\pa A)/2}\,.
  \]
  Now, since $u_h(r)$ is decreasing in $r$, by \eqref{r estimate 2} we find
  \begin{equation}
    \label{Ehr 0}
      u_h(r)\le u_h(C(n)\,m_h^{1/2n})\le C_3\,m_h^{1/2(n-1)}\,,\qquad\forall r\in I\,,
  \end{equation}
  so that, for $h$ large enough,
  \begin{equation}
    \label{Ehr -1}
      \l_h(r)=\frac{1}{(1-u_h(r))^{1/n}}\le 1+C\,u_h(r)\le 1+C\,m_h^{1/2(n-1)}\le 2\,,
  \end{equation}
  and thus
  \begin{equation}
    \label{Ehr 2}
      \hat{E}_h^r\subset B_{x_0,\dist(x_0,\pa A)}\subset A\,,\qquad \forall r\in I\,.
  \end{equation}
  By \eqref{Ehr 1} and \eqref{Ehr 2} we can exploit the minimality of $E_h$ to deduce
  \begin{equation}
    \label{Ehr 3.1}
  \F_{A,\s}(E_h)+\int_{E_h}g\le\F_{A,\s}(\hat{E}_h^r)+\int_{\hat{E}_h^r}g=P(\hat{E}_h^r)+\int_{\hat{E}_h^r}g\,.
  \end{equation}
  We notice that by $g\ge0$, $\hat{E}^r_h=x_0+\l_h(r)(E_h^r-x_0)$, and \eqref{Ehr -1}
  \begin{eqnarray}\nonumber
    \int_{\hat{E}_h^r}g-\int_{E_h}g&\le&\l_h(r)^n\int_{E_h^r}g(x_0+\l_h(y-x_0))\,dy-\int_{E_h^r}g
    \\\nonumber
    &\le&C\,\Big(\|g\|_{C^0(A)}\,m_h\,u_h(r)+\int_{E_h^r}g(x_0+\l_h(y-x_0))-g(y)\,dy\Big)
    \\\nonumber
    &\le&C\,\Big(\|g\|_{C^0(A)}\,m_h\,u_h(r)+\Lip(g)\,|\l_h(r)-1|\,\int_{E_h^r}|y-x_0|\,dy\Big)
    \\\label{Ehr 3.2}
    &\le& C\,m_h\,u_h(r)\,.
  \end{eqnarray}
  At the same time, since $\H^{n-1}( E_h\cap\pa B_{x_h,r})=-m_h\,u_h'(r)$ for a.e. $r>0$, we have
  \begin{eqnarray}\nonumber
  P(\hat{E}_h^r)&=&\l_h(r)^{n-1}\,P(E_h\cap B_{x_h,r})\le (1+C\,u_h(r))\,\big(P(E_h;B_{x_h,r})+m_h\,|u_h'(r)|\big)
  \\\label{Ehr 3.3}
  &\le& P(E_h;B_{x_h,r})+C\,m_h\,\big(|u_h'(r)|+m_h^{-1/n}\,u_h(r)\big)\,,
  \end{eqnarray}
  where we have also used the fact that $P(E_h;B_{x_h,r})\le P(E_h)\le C\,m_h^{(n-1)/n}$. By combining \eqref{Ehr 3.1}, \eqref{Ehr 3.2} and \eqref{Ehr 3.3} we find that for a.e. $r\in I$,
  \[
  P(E_h;A)+\int_{\pa A\cap\pa^* E_h}\s\le P(E_h;B_{x_h,r})+C\,m_h\,\big(|u_h'(r)|+m_h^{-1/n}\,u_h(r)\big)\,,
  \]
  that is
  \begin{eqnarray*}
    C\,m_h\,\big(|u_h'(r)|+m_h^{-1/n}\,u_h(r)\big)\ge P(E_h;A\setminus B_{x_h,r})+\int_{\pa A\cap\pa^* E_h}\s\,,
  \end{eqnarray*}
  for a.e. $r\in I$. By adding up $\H^{n-1}( E_h\cap\pa B_{x_h,r})=m_h\,|u_h'(r)|$ to both sides and using that $B_{x_h,r}\cap\pa A=\emptyset$ implies $\pa A\cap\pa^*E_h=\pa A\cap\pa^*(E_h\setminus B_{x_h,r})$, we conclude
  \begin{eqnarray}\label{Ehr 4}
    C\,m_h\,\big(|u_h'(r)|+m_h^{-1/n}\,u_h(r)\big)&\ge&P(E_h\setminus B_{x_h,r};A)+\int_{\pa A\cap\pa^*E_h}\s=\F_{A,\s}(E_h\setminus B_{x_h,r})\hspace{0.3cm}
    \\\nonumber
    &\ge& C^{-1}\,P(E_h\setminus B_{x_h,r})\ge C^{-1}\,n\om_n^{1/n}\,m_h^{(n-1)/n}\,u_h(r)^{(n-1)/n}\,,
  \end{eqnarray}
  for a.e. $r\in I$, where in the last line we have used \eqref{general inequality} and the isoperimetric inequality. Thanks to \eqref{Ehr 0}, provided $h$ is large enough, one has
  \[
  C\,u_h(r)\le \frac{n\om_n^{1/n}}{2}\,u_h(r)^{(n-1)/n}\,,\qquad \forall r\in I\,,
  \]
  so that, for a.e. $r\in I$,
  \[
  C\,m_h\,|u_h'(r)|\ge \frac{n\om_n^{1/n}}2\,m_h^{(n-1)/n}\,u_h(r)^{(n-1)/n}\,,\quad\mbox{i.e.}\quad [(m_hu_h)^{1/n}]'(r)\le -C\,.
  \]
  By integrating this inequality over $(C(n)\,m_h^{1/2n},r)$ and by \eqref{Ehr 0}, one finds that for every $r\in I$
  \begin{eqnarray*}
  (m_h\,u_h(r))^{1/n}&\le&( m_hu_h(C(n)m_h^{1/2n}))^{1/n}-C\big(r-C(n)m_h^{1/2n}\big)
  \\
  &\le&(C_3\,m_h^{1+[1/2(n-1)]})^{1/n}-C\big(r-C(n)m_h^{1/2n}\big)\,.
  \end{eqnarray*}
  For a suitably large value of $C_*$ we find that
  \[
  u_h(r_*)=0\,,\qquad\mbox{if}\qquad r_*= C(n)\,m_h^{1/2n}+C_* (m_h^{1+[1/2(n-1)]})^{1/n}\,,
  \]
  where $r_*\in I$ provided $h$ is large enough in terms of $C_*$. Since $r_*\le C\,m_h^{1/2n}$ this completes the proof of \eqref{boundedness 2h} in case one.

  \smallskip

  \noindent {\it Case two, $x_0\in\pa A$}: Let $r_0$, $s_0$, $U=U_{x_0}$ and $\phi=\phi_{x_0}$ be as in \eqref{def phix}. For a constant $N\ge 2$ to be determined later on in terms of $n$ and $A$, let us consider $L>0$ (depending on $N$, $n$ and $A$) such that
  \begin{equation}
    \label{r0 s0 L}  B_{x_0,r_0/L}\subset \phi(B_{s_0/N})\,.
  \end{equation}
  As in case one, we define $E_h^r=E_h\cap B_{x_h,r}$ and $u_h(r)=|E_h\setminus B_{x_h,r}|/m_h$ for every $r\in I$, where now -- with $C(n)$ as in the left-hand side of \eqref{r estimate 2} -- we take
  \[
  I=\Big(C(n)m_h^{1/2n},\frac{r_0}{2L}\Big)\,.
  \]
  For $h$ large enough and $r\in I$, by \eqref{r0 s0 L} we  have
  \[
  E_h^r\subset B_{x_h,r}\cap A\subset B_{x_0,r_0/L}\cap A\subset\phi(B_{s_0/N})\cap A\,,
  \]
  so that it makes sense to consider
  \begin{equation}
    \label{Ehr 5}
  \phi^{-1}(E_h^r)\subset B_{s_0/N}\cap H\subset B_{s_0/2}\cap H\,.
  \end{equation}
  We claim that for every $h$ large enough and for every $r\in I$ there exists $\l_h(r)\in(1,2)$ such that
  \[
  \hat{E}_h^r=\phi\big(\l_h(r)\,\phi^{-1}(E_h^r)\big)
  \]
  satisfies
  \begin{equation}
    \label{Ehr 9}
      |\hat{E}_h^r|=m_h\,,\qquad \hat{E}_h^r\subset A\,,\qquad 1\le \l_h(r)\le 1+C\,u_h(r)\,.
  \end{equation}
  Indeed, by \eqref{Ehr 5} we find
  \begin{equation}
    \label{Ehr 6}
      \l\phi^{-1}(E_h^r)\subset H\cap B_{\l s_0/2}\subset H\cap B_{s_0}\subset H\cap U\,,\qquad\forall\l<2\,,
  \end{equation}
  so that $\phi(\l\,\phi^{-1}(E_h^r))$ is defined for every $\l<2$. If $w_h^r(\l)=|\phi(\l\,\phi^{-1}(E_h^r))|$, $\l<2$, then we have
  \[
  w_h^r(1)=|E_h^r|=m_h(1-u_h(r))\le m_h\,,
  \]
  while at the same time, if $1<\l<2$ then by \eqref{jacobian phi estimates} and \eqref{Ehr 6} we find
  \begin{eqnarray}\nonumber
    w_h^r(\l)&=&\int_{\l\,\phi^{-1}(E_h^r)}J\phi\ge(1-C_1s_0)\l^n|\phi^{-1}(E_h^r)|
    =(1-C_1s_0)\l^n\int_{E_h^r}J\phi^{-1}(x)\,dx
    \\\label{Ehr 7}
    &=&(1-C_1s_0)\l^n\int_{E_h^r}\frac{dx}{J\phi(\phi^{-1}(x))}
    \ge\frac{1-C_1s_0}{1+C_1s_0}\,\l^n\,|E_h^r|\,,
  \end{eqnarray}
  that is, by \eqref{Ehr 0} (which still holds for every $r\in I$ even with the new definition of $I$)
  \[
  \frac{w_h^r(\l)}{m_h}\ge \frac{1-C_1s_0}{1+C_1s_0}\,(1-u_h(r))\,\l^n\ge \frac{1-C_1s_0}{1+C_1s_0}\,\big(1-C_3\,m_h^{1/2n}\big)\,\l^n\,.
  \]
  In particular, up to further decreasing the value of $s_0$ and for every $h$ large enough, one can always find $\l_h^*(r)\in(1,2)$ such that
  \[
  w_h^r(\l_h^*(r))> m_h\ge w_h^r(1)\,.
  \]
  Since $w_h^r(\l)=\l^n\int_{\phi^{-1}(E_h^r)}J\phi(\l z)\,dz$ is a continuous function, we find $\l_h(r)\in[1,2)$ such that
 \begin{eqnarray*}
 m_h=w_h^r(\l_h(r))=\l_h(r)^n\,\int_{\phi^{-1}(E_h^r)}J\phi(\l_h(r)z)\,dz
 =
 \l_h(r)^n\,\int_{E_h^r}J\phi(\l_h(r)\phi^{-1}(y))\,J\phi^{-1}(y)\,dy
 \end{eqnarray*}
 so that
 \begin{equation}
   \label{Ehr 8}
    m_h\,u_h(r)=m_h-|E_h^r|=\int_{E_h^r}v(\l_h(r),y)\,dy\,.
 \end{equation}
 where we have set
 \[
 v(\l,y)=\l^n\,J\phi(\l\phi^{-1}(y))\,J\phi^{-1}(y)-1\,,\qquad y\in B_{x_0,2\,r_0}\cap A\,,\l\in[1,2)\,.
 \]
 Now, if $y\in E_h^r$, then $|\phi^{-1}(y)|\le s_0/N$ by \eqref{Ehr 5}, so that for every $\l\in[1,2)$ and $y\in E_h^r$ we find
 \begin{eqnarray*}
   v(\l,y)&\ge&\l^n-1+\l^n\Big(J\phi(\l\phi^{-1}(y))-J\phi(\phi^{-1}(y))\Big)\,J\phi^{-1}(y)
   \\
   &\ge& n\,(\l-1)-2^n\,\Big|J\phi(\l\phi^{-1}(y))-J\phi(\phi^{-1}(y))\Big|\,J\phi^{-1}(y)
   \\
   &\ge& n\,(\l-1)-C\,(\l-1)\,|\phi^{-1}(y)|\ge \Big(n-C_4\,\frac{s_0}{N}\Big)\,(\l-1)\,,
 \end{eqnarray*}
 where $C_4=C_4(n,A)$. Provided we pick $N$ suitably large in terms of $n$ and $A$ we thus find
 \[
 v(\l,y)\ge \l-1\,,\qquad\forall y\in E_h^r\,,\l\in(1,2)\,,
 \]
 and thus deduce from \eqref{Ehr 8} and $|E_h^r|\ge m_h/2$ that if $h$ is large enough, then
 \[
 \l_h(r)\le 1+C\,u_h(r)\,,\qquad\forall r\in I\,.
 \]
 This completes the proof of \eqref{Ehr 9}. By \eqref{Ehr 9} and the minimality of $E_h$ we obtain
 \begin{equation}
   \label{Ehr 10}
    \F_{A,\s}(E_h)+\int_{E_h}g\le\F_{A,\s}(\hat{E}_h^r)+\int_{\hat{E}_h^r}g\,.
 \end{equation}
 On the one hand
 \begin{eqnarray*}
   \int_{\hat{E}_h^r}g-\int_{E_h}g&\le&\int_{E_h^r}\Big(g\big(\phi(\l_h(r)\phi^{-1}(y))\big)\,\big(v(\l_h(r),y)+1\big)-g(y)\Big)\,dy
   \\
   &\le&
   C\,\int_{E_h^r}\Big(g\big(\phi(\l_h(r)\phi^{-1}(y))\big)-g(y)\Big)\,dy
   +\int_{E_h^r}g(y)\,v(\l_h(r),y)\,dy
   \\
   &\le&
   C\,\Lip(g)\int_{E_h^r}\Big|\phi(\l_h(r)\phi^{-1}(y))\big)-y\Big|\,dy
   +\|g\|_{C^0(A)}\,\int_{E_h^r}\,v(\l_h(r),y)\,dy
   \\
   &\le& C\,(\l_h(r)-1)\,|E_h^r|\le C\,m_h\,u_h(r)\,;
 \end{eqnarray*}
 on the other hand, by repeatedly applying the area formula and by \cite[Proposition 17.1]{maggiBOOK}
 \[
 P(\hat{E}_h^r;A)=\l_h(r)^{n-1}\int_{A\cap\pa^*E_h^r}   \big(J^{\l_h(r)\,\phi^{-1}(\pa^*E_h^r)}\phi\big)(\l_h(r)\phi^{-1}(x))\,J^{\pa^*E_h^r}\phi(x)\,d\H^{n-1}(x)\,,
 \]
 and since for every $(n-1)$-rectifiable set $\S\subset B_{x_0,2r_0}\cap A$ we have
 \[
 \big(J^{\phi^{-1}(\S)}\phi\big)(\phi^{-1}(x))J^\S\phi(x)=1\,,\qquad\mbox{for $\H^{n-1}$-a.e. $x\in\S$}\,,
 \]
 by $\l_h(r)\le 1+C\,u_h(r)$ we find
 \[
 P(\hat{E}_h^r;A)\le (1+C\,u_h(r))\,P(E_h^r;A)\,,\qquad\forall r\in I\,,
 \]
 and similarly, since $\s$ is a Lipschitz function,
 \[
 \int_{\pa A\cap\pa^*\hat{E}_h^r}\s\le(1+C\,u_h(r))\,\int_{\pa A\cap\pa^*E_h^r}\s\,,\qquad\forall r\in I\,.
 \]
 By combining these estimates with \eqref{Ehr 10} we thus find
 \[
 P(E_h;A)+\int_{\pa A\cap\pa^* E_h}\s\le (1+C\,u_h(r))\Big(P(E_h^r;A)+\int_{\pa A\cap\pa^* E_h^r}\s\Big)+C\,m_h\,u_h(r)\,,
 \]
 and thus, rearranging terms and for a.e. $r\in I$,
 \begin{eqnarray}\label{mb1}
 P(E_h;A\setminus B_{x_h,r})+\int_{(\pa A\cap\pa^* E_h)\setminus B_{x_h,r}}\s
 \le
 C\Big(m_h\,|u_h'(r)|+u_h(r)\,\F_{A,\s}(E_h^r)+m_h\,u_h(r)\Big)\,.
  \end{eqnarray}
  Since $\F_{A,\s}\le P$ on any subset of $A$ and since perimeter decreased under intersection with convex sets, we have
  \begin{equation}
    \label{mb2}
      \F_{A,\s}(E_h^r)\le P(E_h^r)\le P(E_h)\le C\,m_h^{(n-1)/n}\,,
  \end{equation}
  where in the last inequality we have used \eqref{upper bound PEm}. By combining \eqref{mb1} and \eqref{mb2}, and by adding up $m_h\,|u_h'(r)|=\H^{n-1}(E_h\cap\pa B_{x_h,r})$ to both sides of
  the resulting inequality, we eventually get
 \begin{eqnarray*}
 \F_{A,\s}(E_h\setminus B_{x_h,r})\le C\,m_h\,\big(|u_h'(r)|+\,m_h^{-1/n}\,u_h(r)\big)\,,
 \end{eqnarray*}
 for a.e. $r\in I$, which is analogous to \eqref{Ehr 4}. From here we conclude by arguing exactly as in case one. This completes the proof of \eqref{boundedness 2}.

\medskip

\noindent {\it Step seven}: We prove \eqref{boundedness}. We first notice that it must be $B_{x_m,C\,m^{1/2n}}\cap\pa A\ne\emptyset$, for otherwise by \eqref{boundedness 2} we would get
  \[
  \F_{A,\s}(E_m)=P(E_m)=n\,\om_n^{1/n}\,m^{(n-1)/n}=\Big(\psi(\s_0)+\de\Big)\,\,m^{(n-1)/n}\,,
  \]
  for some positive $\de$ independent of $m$, thus contradicting \eqref{upper bound gamma m}: in particular,
  \begin{equation}
    \label{boundedness reduction 2}
      E_m\subset B_{y_m,C\,m^{1/2n}}\qquad\mbox{for some $y_m\in\pa A$}\,.
  \end{equation}
  This proves the first part of \eqref{boundedness}. We now prove that $\s(y_m)-\s_0\le C\,m^{1/2n}$. Let us set
  \[
  \s_1=\max_{\pa A}\s\,,\qquad c_0=\inf\{\psi'(\tau):\tau\in[\s_0,\s_1]\}\,,
  \]
  so that $c_0>0$ as $[\s_0,\s_1]\cc(-1,1)$ and $\psi'$ is continuous with $\psi'>0$ on $(-1,1)$ thanks to Proposition \ref{lemma psi tau}-(i). In particular,
  \[
  \psi(\tau)\ge\psi(\s_0)+c_0\,(\tau-\s_0)\,,\qquad\forall \tau\in[\s_0,\s_1]\,,
  \]
  and since $\s(y_m)\in[\s_0,\s_1]$, by \eqref{upper bound gamma m} and \eqref{lower bound FAsE small diameter} (which we can apply thanks to \eqref{boundedness reduction 2}), we conclude that we have
  \begin{eqnarray}\nonumber
  (1+C\,m^{1/n})\,\psi(\s_0)\ge\frac{\F_{A,\s}(E_m)}{m^{(n-1)/n}}&\ge&(1-C\,m^{1/2n})\,\psi(\s(y_m))
  \\\label{since now}
  &\ge&(1-C\,m^{1/2n})\Big(\psi(\s_0)+c_0\,(\s(y_m)-\s_0)\Big)\,.
  \end{eqnarray}
  This completes the proof of \eqref{boundedness}.

\medskip

\noindent {\it Step eight}: We prove \eqref{proximity epsilon}. Thanks to \eqref{boundedness} it makes sense to define
\[
F_m=\frac{\phi_{y_m}^{-1}(E_m)}{|\phi_{y_m}^{-1}(E_m)|^{1/n}}\,.
\]
Thanks to Proposition \ref{lemma psi tau}-(ii), it is enough to fix $m_h\to 0^+$, set
\[
E_h=E_{m_h}\,,\qquad F_h=F_{m_h}\,,\qquad y_h=y_{m_h}\,,\qquad \phi_h=\phi_{y_h}\,,
\]
and show that
\begin{equation}
  \label{fine thm 2}
  \lim_{h\to\infty}\F_{H,\s_0}(F_h)=\psi(\s_0)\,.
\end{equation}
We first notice that
  \begin{equation}
    \label{stay put}
      \Big|\frac{|\phi_h^{-1}(E_h)|}{m_h}-1\Big|\le \frac1{m_h}\int_{E_h}|J\phi_{h}^{-1}-1|\le\frac{C}{m_h}\int_{E_h}|y-y_h|\,dy\le C\,m_h^{1/2n}\,,
  \end{equation}
  where we have used $J\phi_h^{-1}(y_h)=1$ and \eqref{boundedness reduction 2}. Similarly, by the area formula on rectifiable sets, one sees that
  \begin{eqnarray*}
    \big|P(\phi_h^{-1}(E_h);H)-P(E_h;A)\big|&\le&\int_{A\cap\pa^*E_h}|J^{\pa^*E_h}\phi_h^{-1}-1|
    \\
    &\le&C\int_{A\cap\pa^*E_h}|y-y_h|\,d\H^{n-1}(y)\le C\,P(E_h;A)\,m_h^{1/2n}
  \end{eqnarray*}
  and, again by \eqref{boundedness},
  \begin{eqnarray*}
    \Big|\s_0\,P(\phi_h^{-1}(E_h);\pa H)-\int_{\pa A\cap\pa^*E_h}\s\Big|&\le&\int_{\pa A\cap\pa^*E_h}|\s_0\,J^{\pa^*E_h}\phi_h^{-1}-\s|
    \\
    &\le&C\int_{\pa A\cap\pa^*E_h}\big(|y-y_h|+\s(y_h)-\s_0\big)\,d\H^{n-1}(y)
    \\
    &\le&C\,P(E_h;\pa A)\,m_h^{1/2n}\,,
  \end{eqnarray*}
  so that, thanks to \eqref{upper bound PEm}
  \[
  |\F_{A,\s}(E_h)-\F_{H,\s_0}(\phi_h^{-1}(E_h))|\le C\,P(E_h)\,m_h^{1/2n}\le C\,m_h^{(n-1)/n}\,m_h^{1/2n}\,.
  \]
  By combining this last estimate with \eqref{stay put} and  \eqref{upper bound gamma m} we thus find that
  \begin{eqnarray*}
    \psi(\s_0)&\le&\F_{H,\s_0}(F_h)=|\phi_h^{-1}(E_h)|^{(1-n)/n}\,\F_{H,\s_0}(\phi_h^{-1}(E_h))
    \\
    &\le&m_h^{(1-n)/n}(1+C\,m_h^{1/2n})\,\big(\psi(\s_0)\,m_h^{(n-1)/n}\,(1+C\,m^{1/n})+C\,m_h^{(n-1)/n}\,m_h^{1/2n}\big)
    \\
    &\le&(1+C\,m_h^{1/2n})\psi(\s_0)\,,
  \end{eqnarray*}
  so that \eqref{fine thm 2} follows.
\end{proof}

\section{$C^{1,\a}$-convergence to the ideal droplet}\label{section convergence to K c1alpha} We now conclude the analysis started in Lemma \ref{thm convergenza a K}. Let us recall that so far we have proved the existence of $m_0>0$ such that if $E_m$ is a minimizer in the variational problem
\[
 \g(m)=\inf\Big\{\F_{A,\s}(E)+\int_E\,g(x)\,dx:E\subset A\,,|E|=m\Big\}\,,
\]
(introduced in \eqref{variational problem m}) with $m<m_0$, then for some $y_m\in\pa A$ and setting
\[
\phi_m=\phi_{y_m}\,,\qquad \l_m=|\phi^{-1}_m(E_m)|^{1/n}\,,\qquad F_m=\frac{\phi_m^{-1}(E_m)}{\l_m}\,,
\]
one has
\begin{equation}
    \label{lambda m ordine}
    \begin{split}
    E_m\subset B_{y_m,C\,m^{1/2n}}\,,\qquad 0\le \s(y_m)-\s_0\le C\,m^{1/2n}\,,
    \\
     \Big|\frac{\l_m}{m^{1/n}}-1\Big|\le C\,m^{1/2n}\,,\qquad\lim_{m\to 0^+}|(F_m-z_m)\Delta K|=0\,,
    \end{split}
\end{equation}
where $z_m\in\pa H$, $K=K(\s_0)$, $\s_0=\min_{\pa A}\s$; see \eqref{boundedness}, \eqref{proximity epsilon} and \eqref{stay put}. Here, as it was set in Notation \ref{notation boundary of A}, $\phi_m$ is a $C^{1,1}$-diffeomorphism between $U_m=U_{y_m}$ and $B_{y_m,2\,r_0}$ such that $\phi_m(0)=y_m$, $\nabla\phi_m(0)$ is a linear isometry of $\R^n$, $B_{s_0}\cc U_m$, and $\phi_m(U_m\cap H)=B_{y_m,2r_0}\cap A$, where $r_0$ and $s_0$ are positive constant depending on $A$, $g$ and $\s$ (whose value will be further decreased in the course of the proof). Moreover, we notice that, as a consequence of \eqref{jacobian phi estimates}, one has that
\begin{equation}
  \label{lambda m ordine 2}
  \|\phi_m\|_{C^{1,1}(U_m)}\le C\,,\qquad \|J\phi_m-1\|_{C^0(B_s)}\le C\,s\qquad\forall s<s_0\,,
\end{equation}
for $C$ depending on $A$ only.  With this situation in mind, we now improve the convergence in volume of $F_m-z_m$ to $K$ into $C^{1,\a}$-convergence. Taking into account Theorem \ref{thm improved convergence to K} it will suffice to show that the sets $F_m$ satisfy uniform almost-minimality conditions with respect to uniformly elliptic functionals.

\begin{lemma}\label{lemma almost minimizers}
  Under the assumptions of Lemma \ref{thm convergenza a K}, and with the notation introduced at beginning of this section, there exists $\l\ge 1$, $\ell\,,\Lambda\ge0$ and $C\,,\rho_0>0$ (depending on $A$, $\s$ and $g$) and elliptic integrands
  \[
  \Psi_m\in\X(B_{C/m^{1/2n}},\l,\ell)
  \]
  such that, for every $m<m_0$, $F_m$ is a $(\Lambda,\rho_0)$-minimizer of $\mathbf{\Psi}_m$ in $(B_{C/m^{1/2n}},H)$, where
  \[
  I_{3\rho_0}(F_m)\cc B_{C/m^{1/2n}}\,.
  \]
\end{lemma}

\begin{proof} Let us set
\[
G_m=\phi^{-1}_m(E_m)\,,\qquad \Om_m=\frac{U_m}{\l_m}\,,\qquad F_m=\frac{G_m}{\l_m}\,.
\]
By \eqref{lambda m ordine} and $\phi_m(0)=y_m$ we have $G_m\subset B_{C\,m^{1/2n}}$, and thus
\begin{equation}
  \label{inclusioni Fm in palle}
F_m\subset B_{C\,m^{1/2n}/\l_m}\subset B_{C/m^{1/2n}}\,,\qquad\forall m<m_0\,,
\end{equation}
where in the last inclusion we have used \eqref{lambda m ordine} again. If we define, for $x\in U_m$, $y\in U_m\cap\pa H$, and $\nu\in S^{n-1}$,
\[
\Phi_m(x,\nu)=|\cof\nabla\phi_m(x)\,\nu|\,,\qquad \s_m(y)=\s(\phi_m(y))\,\Phi_m(y,e_n)\,,\qquad g_m(x)=g(\phi_m(x))\,J\phi_m(x)\,,
\]
then $G_m$ is a minimizer in the variational problem
\begin{equation}
  \label{Gm problem}
  \inf\Big\{\PHI_m(G;H)+\int_{\pa^*G\cap\pa H}\s_m +\int_{G}g_m: G\subset H\cap U_m\,, \int_{G}J\phi_m=m\Big\}\,.
\end{equation}
Indeed, if $G$ is a competitor in \eqref{Gm problem}, then $E=\phi_m(G)$ satisfies $E\subset A$ and $|E|=m$, and, by the area formula,
\begin{eqnarray*}
\F_{A,\s}(E)+\int_{E}g=P(E;A)+\int_{\pa A\cap\pa^*E}\s+\int_{E}g
=\PHI_m(G;H)+\int_{\pa H\cap\pa^*G}\s_m+\int_{G}g_m\,.
\end{eqnarray*}
Similarly, if we set, for $x\in\Om_m$, $y\in\Om_m\cap\pa H$, and $\nu\in S^{n-1}$,
\[
\hat{\Phi}_m(x,\nu)=\Phi_m(\l_m\,x,\nu)\,,\qquad \hat{\s}_m(y)=\s_m(\l_m\,y)\,,\qquad \hat{g}_m(x)=g_m(\l_m\,x)\,,
\]
then $F=G/\l_m\subset \Om_m\cap H$ if and only if $G=\l_m\,F\subset U_m\cap H$, with
\[
\PHI_m(G;H)+\int_{\pa^*G\cap\pa H}\s_m+\int_G\,g_m=\l_m^{n-1}\,\Big(\hat{\PHI}_m(F;H)+\int_{\pa^*F\cap \pa H}\hat{\s}_m\Big)+\l_m^n\int_F\,\hat{g}_m\,.
\]
Hence the fact that $G_m$ is a minimizer in \eqref{Gm problem} implies that $F_m$ is a minimizer in
\begin{equation}
  \label{Fm problem}
  \inf\Big\{\hat{\PHI}_m(F;H)+\int_{\pa^*F\cap\pa H}\hat{\s}_m +\l_m\,\int_{F}\hat{g}_m: F\subset H\cap \Om_m\,, \int_{F}\psi_m=\frac{m}{\l_m^n}\Big\}\,,
\end{equation}
provided one sets
\[
\psi_m(y)=J\Phi_m(\l_m\,y)=|\det\nabla\phi_m(\l_m\,y)|\,,\qquad \forall y\in \Om_m\,.
\]
It is useful to notice that by  \eqref{jacobian phi estimates} and $\l_m\,F_m=G_m\subset B_{C\,m^{1/2n}}$ (recall \eqref{lambda m ordine} and \eqref{inclusioni Fm in palle})
\begin{equation}
  \label{psim star varie}
  \|\psi_m-1\|_{L^\infty(F_m)}\le C\,m^{1/2n}\,,\qquad \|\nabla\psi_m\|_{L^\infty(F_m)}\le C\,\l_m,\qquad \|\psi_m\|_{C^{1,1}(\Om_m)}\le C\,,
\end{equation}
for a constant $C$ depending on $A$ only.

We now want to exploit the minimality of $F_m$ in \eqref{Fm problem} to show the following uniform almost minimality property of $F_m$: for every $F\subset H$ such that $F\Delta F_m\cc B_{z,2\,\rho_0}\cc \Om_m$ for some $z\in H$, one has
\[
\PSI_m(F_m;H)\le\PSI_m(F;H)+\Lambda\,|F\Delta F_m|\,,
\]
for some $\Psi_m\in\X(B_{C/m^{1/2n}},\l,\ell)$. Of course the difficulty here is that such a competitor $F$ may fail to belong to the competition class in \eqref{Fm problem}. However, as $F$ is close to $F_m$, then $\int_F\psi_m$ should be close to $\int_{F_m}\psi_m$. We should be possible to slightly modify $F$ into a new competitor $F'$ with $\int_{F'}\psi_m=\int_{F_m}\psi_m$, while keeping track of the change in surface energy. This is what we do in the next argument, which is loosely based on \cite{Almgren76}, see also \cite[Section 29.6]{maggiBOOK}.

For a value of $\e_0$ to be chosen later depending on $K$ (thus on $\s_0$) only, let us now decrease the value of $m_0$ so to entail
\begin{equation}
  \label{epsilon0}
  |(F_m-z_m)\Delta K|<\e_0\,,\qquad\forall m<m_0\,,
\end{equation}
where $z_m\in\pa H$ as in \eqref{lambda m ordine}. Let us fix $F\subset H$ such that $F\Delta F_m\cc B_{z,2\,\rho_0}\cc \Om_m$ for some $z\in H$, where $\rho_0$ is also to be properly chosen. We pick $x,y\in H\cap\pa K$ (independently from $F_m$) and fix $\tau_0>0$ so that
\[
\ov{B}_{x,\tau_0}\cap \ov{B}_{y,\tau_0}=\emptyset\,,\qquad B_{x,\tau_0}\cup B_{y,\tau_0}\cc B_{s_0/2\,m_0^{1/n}}\cc H\cap\Om_m\quad\forall m<m_0\,.
\]
(The last condition follows from the fact that $B_{s_0}\cc U_m$, so that we can entail $K\cc B_{s_0/\l_m}\cc \Om_m$ for every $m<m_0$, and thanks to the fact that $\l_m/m^{1/n}\to 1$.) Notice that, up to decreasing the value of $\rho_0$, either $\ov{B}_{z-z_m,2\,\rho_0}\cap \ov{B}_{x,\tau_0}=\emptyset$, or $\ov{B}_{z-z_m,2\,\rho_0}\cap \ov{B}_{y,\tau_0}=\emptyset$. Without loss of generality we assume that
\begin{equation}
  \label{disjoint}
  \ov{B}_{z-z_m,2\,\rho_0}\cap \ov{B}_{x,\tau_0}=\emptyset\,,
\end{equation}
and then we fix $T\in C^\infty_c(B_{x,\tau_0};\R^n)$ with
\[
1=\int_{\pa K}T\cdot\nu_K=\int_K\,\Div\,T\,.
\]
(The existence of $T$ follows from $x\in\pa K$, of course.) Correspondingly, we can find $t_0>0$ (depending on $\tau_0$ and $\|T\|_{C^1(\R^n)}$, thus on $K$, thus on $\s_0$) such that for each $|t|<t_0$ the map $f_t=\Id+t\,T$ defines a diffeomorphism of $\R^n$ with $\{f_t\ne\Id\}\cc B_{x,\tau_0}$. Now let us consider
\[
\psi_m^*(w)=\psi_m(w+z_m)=J\Phi_m(\l_m\,(w+z_m))\,,\qquad w\in \Om_m-z_m\,,
\]
and define a smooth function $\g:(-t_0,t_0)\to\R$ by setting
\[
\g(t)=\int_{f_t(F_m-z_m)}\psi_m^*-\int_{F_m-z_m}\psi_m^*=t\,\int_{F_m-z_m}\,\Div(\psi_m^*T)+O(t^2)\,.
\]
Here, thanks to \eqref{psim star varie}, $|O(t^2)|\le C\,t^2$ for a constant $C$ depending on $\|T\|_{C^1(\R^n)}$ and $\|\psi_m^*\|_{C^{1,1}(F_m-z_m)}$, thus on $A$ and $\s_0$, only. Also,
\begin{eqnarray*}
  \Big|1-\int_{F_m-z_m}\,\Div(\psi_m^*T)\Big|&\le&
  \Big|\int_K\Div\,T-\int_{F_m-z_m}\,\psi_m^*\Div T+T\cdot\nabla \psi_m^*\Big|
  \\
  &\le&\|T\|_{C^1(\R^n)}\,\Big(|K\Delta(F_m-z_m)|+\int_{F_m-z_m}|\psi_m^*-1|+|\nabla\psi_m^*|\Big)
  \\
  &\le&\|T\|_{C^1(\R^n)}\,\big(\e_0+m^{1/2n}+\l_m\big)\,,
\end{eqnarray*}
where we have used both \eqref{psim star varie} and \eqref{epsilon0}. In particular, by \eqref{lambda m ordine}, if we further decrease the values of $t_0$, $m_0$ and $\e_0$ depending on  $\|T\|_{C^1(\R^n)}$, then we find
\begin{equation}
  \label{implicit}
  \g(0)=0\,,\qquad \g'(0)\ge\frac12\,,\qquad \Lip(\g';(-t_0,t_0))\le C\,\,.
\end{equation}
By \eqref{implicit}, up to further decreasing the value of $t_0$, we can find $\eta_0>0$ (depending on $A$, $g$ and $\s$) such that $\g^{-1}$ is well-defined on $(-2\eta_0,2\eta_0)$ with
\begin{equation}
  \label{gF meno uno}
  |\g^{-1}(v)|\le C\,|v|\,,\qquad\mbox{for every $|v|\le2\eta_0$}\,.
\end{equation}
Recalling that $F\subset H$ with $F\Delta F_m\cc B_{z,2\,\rho_0}\cc \Om_m$, we set
\[
v=\int_{F_m-z_m}\psi_m^*-\int_{F-z_m}\psi_m^*\,.
\]
Up to further decreasing the value of $\rho_0$ we find
\begin{equation}
  \label{stimetta}
|v|\le\|\psi_m\|_{C^0(\Om_m)}\,|F\Delta F_m|\le C\,\om_n\,\rho_0^n<2\,\eta_0\,,
\end{equation}
so that by \eqref{gF meno uno} we can compute $\g^{-1}(v)$ and correspondingly define $F'\subset H$ by letting
\[
F'-z_m=\Big((F-z_m)\cap B_{z-z_m,2\,\rho_0}\Big)\cup\Big(f_{\g^{-1}(v)}(F_m-z_m)\setminus B_{z-z_m,2\,\rho_0}\Big)\,.
\]
Notice that, by construction, $F'$ and $F_m$ are equal on $H\setminus(B_{z,2\rho_0}\cup B_{x+z_m,\tau_0})$, so that
\begin{eqnarray*}
  \int_{F'}\psi_m-\int_{F_m}\psi_m&=&
  \int_{F\cap B_{z,2\rho_0}}\psi_m-\int_{F_m\cap B_{z,2\rho_0}}\psi_m
  \\
  &&+\int_{f_{\g^{-1}(v)}(F_m-z_m)\cap B_{x,\tau_0}}\psi_m^*-
  \int_{(F_m-z_m)\cap B_{x,\tau_0}}\psi_m^*
  \\
  &=&
  \int_F\psi_m-\int_{F_m}\psi_m+\int_{f_{\g^{-1}(v)}(F_m-z_m)}\psi_m^*-
  \int_{F_m-z_m}\psi_m^*
  \\
  &=&
  \int_F\psi_m-\int_{F_m}\psi_m+\g(\g^{-1}(v))=0\,.
\end{eqnarray*}
Hence, $F'$ is a competitor for \eqref{Fm problem}, and since $B_{x,\tau_0}\cc H$ (thus $B_{x+z_m,\tau_0}\cc H$ too) by comparing $F_m$ to $F'$ we find
\begin{eqnarray}\label{treno4}
  &&\hat{\PHI}_m(F_m;B_{x+z_m,\tau_0}\cup (B_{z,2\rho_0}\cap H))+\int_{B_{z,2\rho_0}\cap\pa^*F_m\cap\pa H}\hat{\s}_m
  \\\nonumber
  &\le&
  \hat{\PHI}_m(F';B_{x+z_m,\tau_0}\cup (B_{z,2\rho_0}\cap H))+\int_{B_{z,2\rho_0}\cap\pa^*F\cap\pa H}\hat{\s}_m +\l_m\,\|\hat{g}_m\|_{C^0(\Om_m)}\,|F'\Delta F_m|\,,
\end{eqnarray}
where in writing the second term on the second line we have taken into account that $\pa^*F$ and $\pa^*F'$ are equal on $B_{z,2\rho_0}\cap\pa H$. In order to exploit \eqref{treno4} it is useful to show that
\begin{eqnarray}\label{treno3}
  |(F'\Delta F_m)\cap B_{x+z_m,\tau_0}|+|\hat{\PHI}_m(F_m;B_{x+z_m,\tau_0})-\hat{\PHI}_m(F';B_{x+z_m,\tau_0})|\le C\,|F\Delta F_m|\,.
\end{eqnarray}
To begin with, by \cite[Lemma 17.9]{maggiBOOK}
\begin{eqnarray}\nonumber
  |(F'\Delta F_m)\cap B_{x+z_m,\tau_0}|&\le&\big| f_{\g^{-1}(v)}(F_m-z_m)\Delta(F_m-z_m)\big|
  \\\nonumber
  &\le& C\,|\g^{-1}(v)|\,P(F_m-z_m;B_{x,\tau_0})\le\,C\,|v|\,P(F_m;H)\,;
\end{eqnarray}
since $P(F_m;H)=\l_m^{1-n}\,P(G_m;H)\le C\,m^{(1-n)/n}\,P(E_m)$, \eqref{upper bound PEm} gives $P(F_m;H)\le C$ , and thus by \eqref{psim star varie} and the definition of $v$
\begin{eqnarray}
  \nonumber
  |(F'\Delta F_m)\cap B_{x+z_m,\tau_0}|&\le& C\,\Big|\int_{F_m}\psi_m-\int_{F}\psi_m\Big|\le C\,\|\psi_m\|_{C^0(\Om_m)}\,|F_m\Delta F|
  \\\label{treno1}
  &\le&C\,|F_m\Delta F|\,.
\end{eqnarray}
This proves part of \eqref{treno3}. To complete the proof of \eqref{treno3}, given $p\in\pa^*F_m$, let $\{\tau_i(p)\}_{i=1}^{n-1}$ denote an orthonormal basis of the approximate tangent space of $\pa^*F_m$ at $p$, chosen so that $\nu_{F_m}(p)=\wedge_{i=1}^{n-1}\tau_i(p)$. Since $\hat{\Phi}_m$ has finite Lipschitz constant on $\Om_m\times\R^n$, one has
\[
\Big|\hat\Phi_m\Big(f_t(p),\bigwedge_{i=1}^{n-1}df_t(p)[\tau_i(p)]\Big)-\hat\Phi_m(p,\nu_{F_m}(p))\Big|\le C\,|t|\,,
\]
for every $|t|<t_0$. By the area formula, setting $t=\g^{-1}(v)$ for the sake of brevity, we have
\begin{eqnarray}\nonumber
&&|\hat{\PHI}_m(F_m;B_{x+z_m,\tau_0})-\hat{\PHI}_m(F';B_{x+z_m,\tau_0})|
\\\nonumber
&\le&\int_{B_{x+z_m,\tau_0}\cap\pa^*F_m}
\Big|\hat\Phi_m\Big(f_t(p),\bigwedge_{i=1}^{n-1}df_t(p)[\tau_i(p)]\Big)-\hat\Phi_m(p,\nu_{F_m}(p))\Big|\,d\H^{n-1}(p)
\\\label{treno2}
&\le& C\,|\g^{-1}(v)|\,P(F_m;B_{x+z_m,\tau_0})\le C\,|F\Delta F_m|\,,
\end{eqnarray}
where in the last inequality we have argued as in the proof of \eqref{treno1}. This proves \eqref{treno3}, which combined with \eqref{treno4} gives us
\begin{eqnarray}\label{treno5}
  \hat{\PHI}_m(F_m;H)+\int_{\pa^*F_m\cap\pa H}\hat{\s}_m
  \le
  \hat{\PHI}_m(F;H)+\int_{\pa^*F\cap\pa H}\hat{\s}_m +C\,|F\Delta F_m|\,.
\end{eqnarray}
Finally, let $\s_m^*:\R^n\to\R$ be a Lipschitz function such that $\s_m^*=\hat\s_m$ on $\pa H$ and
\begin{equation}
  \label{sigma ext}
  \Lip(\s_m^*)=\Lip(\hat\s_m;\pa H)\,.
\end{equation}
(There is a huge freedom in the choice of $\s_m^*$ and we shall exploit it later.) By the divergence theorem
\begin{eqnarray*}
\int_{\pa^*F\cap \pa H}\hat\s_m& = &\int_{\pa^*F\cap \pa H}(-e_n\,\s_m^*)\cdot(-e_n)=\int_{\pa^*F\cap H}\s_m^*\, e_n\cdot\nu_{ F}+\int_{ F}\Div(-e_n\s_m^*)\\
& = &\int_{H\cap\pa  F}\s_m^*\,e_n\cdot\nu_{ F}+\int_{ F}(-e_n)\cdot\nabla\s_m^*\,,
\end{eqnarray*}
so that
\[
\hat{\PHI}_m( F;H)+\int_{\pa^*F\cap\pa H}\hat\s_m =\int_{H\cap\pa^*F}(\hat{\Phi}_m(x,\nu_F)+\s_m^*e_n\cdot\nu_{ F})-\int_{ F}e_n\cdot\nabla\s_m^*\,.
\]
In conclusion, if we set for $x\in\Om_m$ and $\nu\in \R^n$,
\begin{equation}
  \label{psim integrand}
  \Psi_m(x,\nu)=\hat{\Phi}_m(x,\nu)+\s_m^*(x)\,e_n\cdot\nu=|{\rm cof}\nabla\phi_m(\l_m x)\nu|+\s_m^*(x)\,e_n\cdot\nu\,,
\end{equation}
then \eqref{treno5} implies
\begin{equation}
  \label{Fm minimality 2}
\PSI_m(F_m;H)\le \PSI_m(F;H)+\Lambda\,|F\Delta F_m|\,,
\end{equation}
whenever $F\subset H$ is such that $\diam(F\Delta F_m)<2\rho_0$. We now claim that
\[
\Psi_m\in\X(B_{C/m^{1/2n}},\l,\ell)\,,\qquad\forall m<m_0\,,
\]
and for suitable constants $\l$ and $\ell$. It is clear from \eqref{psim integrand} that $\Psi_m$ is lower semicontinuous on $\Om_m\times\R^n$ with $\Psi_m(x,\cdot)$ convex, one homogeneous and with restriction to $S^{n-1}$ of class $C^{2,1}$ for every fixed $x\in\Om_m$. Similarly, one easily deduces from \eqref{lambda m ordine 2} and \eqref{sigma ext} an upper bound on $\Psi_m$ and Lipschitz-type bounds on $\Psi_m(\cdot,\nu)$, $\nabla\Psi_m(\cdot,\nu)$, $\Psi_m(x,\cdot)$, $\nabla\Psi_m(x,\cdot)$ and $\nabla^2\Psi_m(x,\cdot)$ on $\Om_m\times\R^n$ (recall that here $\nabla$ and $\nabla^2$ denote derivatives in the $\nu$ variable). The restriction of $x$ to $B_{C/m^{1/2n}}$, and a more precise choice of $\s_m^*$, come to play in order to check that
\[
\Psi_m(x,\nu)\ge\frac1\l\,,
\qquad
\big|\nabla^2\Psi_m(x,\nu)[\tau,\tau]\big|\ge\frac{|\tau|^2}\l\,,\qquad\forall \tau\in\nu^\perp\,,
\]
whenever $x\in B_{C/m^{1/2n}}$ and $\nu\in S^{n-1}$. Indeed, by \eqref{marina1} and \eqref{lambda m ordine} we have
\begin{equation}
  \label{elliptic check 1}
  |\cof(\nabla\phi_m(\l_m\,x))\nu|\ge|\cof(\nabla\phi_m(0))\nu|-C\,\l_m\,|x|=1-C\,\l_m\,|x|\ge 1-C\,m^{1/2n}\,,
\end{equation}
for every $x\in B_{C/m^{1/2n}}$ and $\nu\in S^{n-1}$. By an analogous argument
\begin{equation}
\label{sigmahat bounds}
|\hat{\s}_m(x)-\s_0|=|\s(\phi_m(\l_m x))|\cof(\nabla\phi_m(\l_m\,x))e_n|-\s(\phi_m(0))|\le C\,\l_m\,|x|\le C\,m^{1/2n}\,,
\end{equation}
for every $x\in B_{C/m^{1/2n}}\cap\pa H$. The idea is then the following: having in mind \eqref{inclusioni Fm in palle}, we first pick $C$ so large to enforce $I_{3\rho_0}(F_m)\cc B_{C/m^{1/2n}}$ (in this way, all the variations considered in \eqref{Fm minimality 2} are contained in the domain of ellipticity of $\Psi_m$); next, we define $\s_m^*$ as a Lipschitz-preserving extension to $\R^n$ of the restriction to $B_{C/m^{1/2n}}$ of $\hat\s_m$, which is then truncated so to preserve the bounds \eqref{sigmahat bounds}, so that
\begin{equation}
\label{sigmastar bounds}
|\s^*_m(x)-\s_0|\le C\,m^{1/2n}\,,\qquad\forall x\in\R^n\,.
\end{equation}
By combining \eqref{elliptic check 1} and \eqref{sigmahat bounds} we find
\[
\Psi_m(x,\nu)\ge 1-|\s_0|-C\,m^{1/2n}\ge\frac{1-|\s_0|}{2}\,,
\]for every $x\in B_{C/m^{1/2n}}$ and $\nu\in S^{n-1}$,
provided $m_0$ is small enough. The Hessian bound is even simpler (as it does not involve the adhesion coefficient $\s$), and so the proof is complete.
\end{proof}

We are now ready to complete the proof of Theorem \ref{thm main}.

\begin{proof}[Proof of Theorem \ref{thm main}] Our starting point is given by Lemma \ref{thm convergenza a K} and Lemma \ref{lemma almost minimizers}.
The fact that $F_m$ is a $(\Lambda,r_0)$-minimizer of $\PSI_m$ in $(B_{C/m^{1/2n}},H)$ with $I_{3\rho_0}(F_m)\cc B_{C/m^{1/2n}}$ implies the existence of $\rho_1=\rho_1(\rho_0,\l,\Lambda)>0$ and $\k=\k(n,\l)>0$ such that if $r\le \rho_1$, then
\[
  \begin{split}
    |F_m\cap B_{x,r}|\ge\k\,|B_{x,r}\cap H|\,,&\qquad\forall x\in \ov{H}\cap\pa F_m\,,
    \\
    |F_m\cap B_{x,r}|\le (1-\k)\,|B_{x,r}\cap H|\,,&\qquad\forall x\in\ov{H}\cap\pa(H\setminus F_m)\,;
  \end{split}
\]
see, e.g. \cite[Lemma 2.8]{dephilippismaggiARMA}. In particular, if we set
\[
F_m^*=F_m-z_m\,,
\]
then for every $r\le\rho_1$
\begin{equation}\label{density estimates}
  \begin{split}
    |F_m^*\cap B_{x,r}|\ge\k\,|B_{x,r}\cap H|\,,&\qquad\forall x\in \ov{H}\cap\pa F_m^*\,,
    \\
    |F_m^*\cap B_{x,r}|\le (1-\k)\,|B_{x,r}\cap H|\,,&\qquad\forall x\in\ov{H}\cap\pa(H\setminus F_m^*)\,.
  \end{split}
\end{equation}
The density estimates \eqref{density estimates} combined with the $L^1$-convergence of $ F_m^*$ to $K$ imply the convergence of $M_m^*=\ov{H\cap\pa F_m^*}$ to $M_0=\ov{H\cap\pa K}$ in Hausdorff distance. Indeed, let $x\in M_m^*$ be such that
\[
\dist(x,M_0)=\sup_{y\in M_m^*}\dist(y,M_0)=\beta\,,
\]
and let $r=\min\{\beta,\rho_1\}$. Since $M_m^*\subset\ov{H}\cap\pa F_m^*$ and $M_m^*\subset \ov{H}\cap\pa(H\setminus F_m^*)$, we have that both estimates in \eqref{density estimates} hold at $x$ with $r=\min\{\beta,\rho_1\}$. Since $B_{x,\beta}\cap M_0=\emptyset$, we have $H\cap B_{x,r}\subset H\setminus K$ or $H\cap B_{x,r}\subset K$: in the first case
\[
\k|B_{x,r}\cap H|\le| F_m^*\cap B_{x,r}|\le| F_m^*\setminus K|\le\e_0\,,
\]
while in the second case one finds
\[
\k|B_{x,r}\cap H|\le|(H\cap B_{x,r})\setminus F_m^*|\le |K\setminus F_m^*|\le\e_0\,;
\]
in both cases, $r^n\le C\,\e_0$, so that, up to further decreasing the value of $\e_0$ one has $\beta^n\le C\,|K\Delta F_m^*|$. Since density estimates analogous to \eqref{density estimates} holds with $K$ in place of $F_m$ (for constants $\rho_1$ and $\k$ depending on $n$ and $\s_0$ only) we conclude that
\begin{equation}
  \label{hd Fm}
  \hd(M_0,M_m^*)^n\le C\,|F_m^*\Delta K|\,,\qquad\forall m<m_0\,.
\end{equation}
An important consequence of \eqref{hd Fm} is that provided $\e_0$ is small enough, then one has
\begin{equation}
  \label{inc}
\frac{K}2\subset F_m^*\,,\qquad\forall m< m_0\,.
\end{equation}
By exploiting this inclusion together with \eqref{upper bound gamma m} we see that $F_m^*$ satisfies \eqref{hp technical}, and thus conclude by means of Proposition \ref{lemma psi tau}-(iii) that
\[
C\,m^{1/n}\ge\F_{H,\s}(F_m^*)-\psi(\tau)\ge c(n,\tau)\,\inf_{z\in\pa H}|(F_m^*-z)\Delta K|^2\,.
\]
Since, by definition of $z_m$,
\[
|F_m^*\Delta K|=\inf_{z\in\pa H}|(F_m-z)\Delta K|=\inf_{z\in\pa H}|(F_m^*-z)\Delta K|\,,
\]
by also taking \eqref{hd Fm} we have found the quantitative estimates
\begin{equation}
  \label{quantitative Fmstar}
  |F_m^*\Delta K|\le C\,m^{1/2n}\,,\qquad \hd(M_0,M_m^*)\le C\,m^{1/2n^2}\,,\qquad\forall m<m_0\,.
\end{equation}
Let us now go back to a simpler consequence of \eqref{hd Fm}, namely
\begin{equation}
  \label{diameter Fmstar}
  I_{3\rho_0}(F_m^*)\cc B_C\,.
\end{equation}
By setting $\Psi_m^*(x,\nu)=\Psi_m(x+z_m,\nu)$, we find that each $F_m^*$ is a $(\Lambda,\rho_0)$-minimizer of $\PSI_m^*$ in $(B_C,H)$ where $\Psi_m^*\in\X(B_C,\l,\ell)$. We are thus in the position to apply Theorem \ref{thm improved convergence to K} to deduce that, for every $m<m_0$, $M_m^*$ is a compact connected orientable $C^{1,\a}$-hypersurface with boundary for every $\a\in(0,1)$ and there exists a diffeomorphism $f_m^*:M_0\to M_m^*$ such that
\[
\|f_m^*\|_{C^{1,\a}(M_0)}\le C_\a\,,\qquad\lim_{m\to 0}\|f_m^*-\Id\|_{C^1(M_0)}=0\,.
\]
Notice that, by construction, this last limit relation does not depend on the specific family of minimizers $E_m$ that we are considering, but just on $A$, $g$ and $\s$.

We now complete the proof of the theorem. We first notice that \eqref{diameter Fmstar} implies $\diam(F_m)\le C$, and thus, thanks to \eqref{lambda m ordine},
\[
\diam(\phi_m^{-1}(E_m))\le C\,\l_m\le C\,m^{1/n}\,.
\]
Since the maps $\phi_m^{-1}$ are uniformly Lipschitz, we conclude that $\diam(E_m)\le C\,m^{1/n}$. In turn, up to change our choice of $y_m\in\pa A$, we can improve the factor $m^{1/2n}$ in \eqref{boundedness reduction 2} into $m^{1/n}$ and repeat the above arguments (starting from step seven in the proof of Lemma \ref{thm convergenza a K}, continuing with the whole proof of Lemma \ref{lemma almost minimizers}, and including the current proof up to this point) using the more precise information
\begin{equation}
  \label{improve}
        E_m\subset B_{y_m,C\,m^{1/n}}
\end{equation}
in place of \eqref{boundedness reduction 2}. In this way we improve $\s(y_m)-\s_0\le C\,m^{1/2n}$ to
\begin{equation}
  \label{improve sigma}
  \s(y_m)-\s_0\le C\,m^{1/n}\,,
\end{equation}
we improve $|m^{-1/n}\l_m-1|\le C\,m^{1/2n}$ in \eqref{lambda m ordine} to
\begin{equation}
  \label{lambda m ordine improved}
  \Big|\frac{\l_m}{m^{1/n}}-1\Big|\le C\,m^{1/n}\,,
\end{equation}
(by the same argument used in \eqref{stay put}) and we replace $F_m\subset B_{C/m^{1/2n}}$ with
\[
F_m\subset B_C\,.
\]
By combining this last inclusion with \eqref{inc} (which gives $0\in F_m^*=F_m-z_m$) we find that $|z_m|\le C$, and thus
\[
z_m\,\l_m\in B_{C\,m^{1/n}}\cap\pa H\subset B_{s_0}\cap\pa H\subset \phi_m^{-1}(B_{2r_0}(y_m)\cap\pa A)\,.
\]
In particular there exists $p_m\in\pa A$ such that
\[
z_m\,\l_m=\phi_m^{-1}(p_m)\,,
\]
and by $\phi_m^{-1}(y_m)=0$ we find $|y_m-p_m|\le C\,m^{1/n}$. In this way \eqref{improve} and \eqref{improve sigma} give us
\begin{equation}
  \label{improve pm}
  E_m\subset B_{p_m,C\,m^{1/n}}\,,\qquad 0\le\s(p_m)-\s_0\le C\,m^{1/n}\,,
\end{equation}
that is \eqref{main theorem small diameter}. We now prove that \eqref{main theorem hd} holds with the linear isometry  $S_m=\nabla\phi_m^{-1}(y_m)$ ($S_m$ is a linear isometry since $S_m=R_m^{-1}$ for $R_m=\nabla\phi_m(0)$). Indeed, by \eqref{quantitative Fmstar} and with the notation $M_m=\ov{A\cap\pa E_m}$, we find
\begin{eqnarray}\label{lu}
  C\,m^{1/2n^2}\ge\hd(M_0,M_m^*)=\hd\Big(M_0,\frac{\phi_m^{-1}(M_m)-\phi_m^{-1}(p_m)}{\l_m}\Big)\,.
\end{eqnarray}
By
\[
\|\phi_m^{-1}-\phi_m^{-1}(p_m)-\nabla\phi_m^{-1}(p_m)(\cdot-p_m)\|_{C^0(B_{p_m,C\,m^{1/n}})}\le C\,m^{2/n}\,,
\]
and thanks to \eqref{lambda m ordine improved},
\[
\hd\Big(\frac{\phi_m^{-1}(M_m)-\phi_m^{-1}(p_m)}{\l_m},\frac{\nabla\phi_m^{-1}(p_m)(M_m-p_m)}{\l_m}\big)\le C\,m^{1/n}\,,
\]
so that by \eqref{lu} and  by linearity
\[
C\,m^{1/2n^2}\ge\hd\Big(M_0,\nabla\phi_m^{-1}(p_m)\Big(\frac{M_m-p_m}{\l_m}\Big)\Big)\,.
\]
By $|\nabla\phi_m^{-1}(p_m)-S_m|\le C\,m^{1/n}$, \eqref{diameter Fmstar} and \eqref{lambda m ordine improved} we find
\[
\hd\Big(S_m\Big(\frac{M_m-p_m}{m^{1/n}}\Big),\nabla\phi_m^{-1}(p_m)\Big(\frac{M_m-p_m}{\l_m}\Big)\Big)\le C\,m^{1/n}\,,
\]
and thus, in conclusion,
\[
C\,m^{1/2n^2}\ge\hd\Big(M_0,S_m\Big(\frac{M_m-p_m}{m^{1/n}}\Big)\Big)\,,
\]
that is \eqref{main theorem hd}. We are thus left to prove the existence of the diffeomorphism $f_m$ between $M_0$ and $M_m$ such that \eqref{diffeo convergence} holds. To this end, let us define
\[
f_m(x)=\phi_m\big(\l_m\,(f_m^*(x)+z_m)\big)=\phi_m(\l_m\,g_m(x))\,,\qquad x\in M_0\,,
\]
where
\[
g_m(x)=f_m^*(x)+z_m\,,\qquad x\in M_0\,.
\]
By construction, $f_m$ is a $C^{1,\a}$-diffeomorphism between $M_0$ and $M_m$, so that $M_m$ is a connected $C^{1,\a}$-hypersurface with boundary such that $\bd(M_m)\subset\pa A$. Now $|z_m|\le C$, while \eqref{diameter Fmstar} gives us $\|f_m^*\|_{C^0(M_0)}\le C$, so that
\begin{equation}
  \label{gm c0}
  \|g_m\|_{C^0(M_0)}\le C\,.
\end{equation}
If $v_m=R_m\,z_m$ (so that $|v_m|\le C$), then we have
\begin{eqnarray*}
f_m(x)-\big(y_m+\l_m(v_m+R_m\,x)\big)
&=&
\phi_m\big(\l_m\,g_m(x)\big)-\phi_m(0)-\nabla\phi_m(0)[\l_m\,g_m(x)]
\\
&&+\nabla\phi_m(0)[\l_m\,(f_m^*(x)-x)]
\end{eqnarray*}
and thus, by $\|\phi_m\|_{C^{1,1}(U_m)}\le C$,
\begin{equation}
  \label{C0}
  \big\|f_m-\big(y_m+\l_m(v_m+R_m\,x)\big)\big\|_{C^0(M_0)}\le C\,\l_m\Big(\l_m+\|f_m^*-\Id\|_{C^0(M_0)}\Big)\,.
\end{equation}
Similarly, if $\tau\in T_xM_0$, then
\[
\nabla^{M_0}f_m(x)[\tau]=\nabla\phi_m\big(\l_m\,g_m(x)\big)[\l_m\,\nabla f_m^*(x)\tau]\,,
\]
so that
\begin{eqnarray*}
  \nabla^{M_0}f_m(x)[\tau]-\l_m\,R_m\tau&=&\l_m\,\nabla\phi_m\big(\l_m\,g_m(x)\big)\,\big[\nabla f_m^*(x)[\tau]-\tau\big]
  \\
  &&+\l_m\,\Big(\nabla\phi_m\big(\l_m\,g_m(x)\big)-\nabla\phi_m(0)\Big)[\tau]\,,
\end{eqnarray*}
and hence
\begin{equation}
  \label{C1}
  \|\nabla^{M_0}f_m-\l_m\,R_m\|_{C^0(M_0)}\le C\,\l_m\,\Big(\l_m+\|f_m^*-\Id\|_{C^1(M_0)}\Big)\,.
\end{equation}
By \eqref{lambda m ordine improved}, \eqref{C0} and \eqref{C1} we finally find
\[
\big\|f_m-\big(y_m+m^{1/n}(v_m+R_m\,x)\big)\big\|_{C^1(M_0)}\le C\,m^{1/n}\,\Big(m^{1/n}+\|f_m^*-\Id\|_{C^1(M_0)}\Big)=o(m^{1/n})\,,
\]
where this limit relation depends on $A$, $\s$ and $g$ only, but not on the particular family of minimizers $E_m$ under consideration. Since $x\mapsto v_m+R_m x$ is an isometry of $\R^n$ we have completed the proof of \eqref{diffeo convergence}.
\end{proof}

\begin{remark}\label{rmk finale}
  {\rm In Remark \ref{remark intro} we claimed that $\g(m)=\psi(\s_0)\,m^{(n-1)/n}\,(1+O(m^{1/n}))$. By taking into account the upper bound \eqref{upper bound gamma m}, we are left to show that
  \[
  \F_{A,\s}(E_m)\ge\psi(\s_0)\,m^{(n-1)/n}\,(1+O(m^{1/n}))\,.
  \]
  Going back to the proof of Lemma \ref{thm convergenza a K}, we notice that \eqref{since now} can be improved into
  \[
  \frac{\F_{A,\s}(E_m)}{m^{(n-1)/n}}\ge(1-C\,m^{1/1n})\,\psi(\s(p_m))\,,
  \]
  since now we are replacing \eqref{boundedness reduction 2} with $E_m\subset B_{p_m,C\,m^{1/n}}$. We conclude as $\psi(\s(p_m))\ge\psi(\s_0)$.}
\end{remark}

\bibliography{references}
\bibliographystyle{is-alpha}

\end{document}